\documentclass[twocolumn]{autart}    

\date{}
\usepackage{amsmath,amssymb,amsfonts}                               
\usepackage{mathrsfs}
\usepackage{color} 
\usepackage{apacite}
\usepackage{subfigure}
\usepackage{graphicx,epsfig}

\usepackage{xpatch}

\usepackage{booktabs}

\makeatletter
\def\ps@pprintTitle{%
	\let\@oddhead\@empty  
	\let\@evenhead\@empty  
	\def\@oddfoot{\reset@font\hfil\thepage\hfil}  
	\let\@evenfoot\@oddfoot  
}
\makeatother

\makeatletter
\ExplSyntaxOn
\cs_new:Npn \bibColoredItems #1#2
{
	\clist_map_inline:nn {#2} { \cs_new:cpn {bib@colored@##1} {#1} } 
}
\ExplSyntaxOff
\newcommand\bib@setcolor[1]{%
	\ifcsname bib@colored@#1\endcsname
	\expandafter\color\expandafter{\csname bib@colored@#1\endcsname}
	\else
	\normalcolor
	\fi
}
\xpatchcmd\@bibitem
{\item}
{\bib@setcolor{#1}\item}
{}{\PatchFailed}

\xpatchcmd\@lbibitem
{\item}
{\bib@setcolor{#2}\item}
{}{\PatchFailed}
\makeatother

\usepackage{algorithm}
\usepackage{algorithmicx}
\usepackage{algpseudocode}

\usepackage{appendix}

\newtheorem{definition}{Definition}
\newtheorem{theorem}{Theorem}

\newtheorem{Proposition}{Proposition}

\newtheorem{Assumption}{Assumption}

\newtheorem{Remark}{Remark}

\newtheorem{Problem}{Problem}

\newcommand{\norm}[1]{\left\|#1\right\|}
\newcommand{\normF}[1]{\norm{#1}_F}

\newcommand{\T}{\top}

\begin{document}

\begin{frontmatter}

\title{\thanksref{footnoteinfo}} 

\thanks[footnoteinfo]{This paper was not presented at any IFAC 
meeting.\\ $*$ Corresponding author.}

\author[Qingdao,Zibo]{Jing Guo}\ead{guojing8299@163.com},          
\author[East]{Xiushan Jiang}$^{,*}$\ead{ x\_sjiang@163.com},	        
\author[Qingdao]{Weihai Zhang}\ead{w\_hzhang@163.com}       

\address[Qingdao]{College of Electrical Engineering and Automation, Shandong University of Science and Technology, Qingdao 266590, China}                                               
\address[Zibo]{School of Mathematics and Statistics,
	Shandong University of Technology, Zibo 250049,
	China} 
\address[East]{College of New Energy, China University of Petroleum (East China), Qingdao 266580, China}                                                    
          

\begin{abstract}                          
	This paper investigates a model-free solution to the stochastic linear quadratic regulation (LQR) problem for linear  discrete-time systems with both multiplicative and additive noises. We formulate the stochastic LQR problem as a nonconvex optimization problem and rigorously analyze its dual problem structure.  By exploiting the  inherent convexity of the dual problem  and analyzing Karush-Kuhn-Tucker  conditions with respect to optimality in convex optimization, we establish an explicit relationship between the optimal point of the dual problem and the   parameters of the associated Q-function. This theoretical insight, combined with the technique of the matrix direct sum, makes it possible to develop a novel model-free sample-efficient, non-iterative semidefinite programming  algorithm that directly estimates optimal control gain without requiring  an initial stabilizing controller, or  noises measurability. The robustness of the model-free SDP method to errors  is investigated. Our approach provides a new optimization-theoretic framework for understanding Q-learning algorithms while advancing the theoretical foundations of reinforcement learning in stochastic optimal control. Numerical validation on a pulse-width modulated inverter system demonstrates the algorithm's effectiveness, particularly in achieving a single-step non-iterative solution without hyper-parameter tuning.
\end{abstract}

\end{frontmatter}

\section{Introduction} 
Optimal  control \cite{lewis2012optimal, yong2012stochastic} is an important branch of modern control theory, which aims to find the optimal control policy for dynamic system by optimising the performance index. The linear quadratic regulation (LQR) problem, which has a wide engineering background, is a typical optimal control problem, and its quadratic performance index is actually a comprehensive consideration of the transient performance, steady state performance and control energy constraints of the system in classical control theory. Such a problem can be solved by such well-established methods as   Bellman's dynamic programming \cite{bellman1966dynamic, bertsekas2019reinforcement}. With the development of convex optimization and semidefinite programming (SDP) \cite{boyd2004convex, vandenberghe1996semidefinite}, many research works  have reinvestigated the LQR problem from the perspective of convex optimization, see \cite{yao2001stochastic} and the references therein. 
The above methods for solving LQR usually require complete system information. And in fact, accurate dynamical models of most practical complex systems are difficult to obtain. To avoid the requirement of system dynamics in controller designs, reinforcement learning (RL) \cite{sutton1998reinforcement, bertsekas2019reinforcement}, a subfield of machine learning that tackles the problem of how an agent learns optimal policy to minimize the cumulative cost by interacting with   unknown environments, has been applied to the LQR problem for deterministic systems. Studies on the design of model-free controllers for LQR by means of the methods of RL, such as temporal difference \cite{sutton1998reinforcement}, Q-learning \cite{watkins1992q}  and so  on, can be found in \cite{bradtke1994adaptive, lewis2009reinforcement}.  However, most of the current RL methods are weakly extendible or lacks in theoretical assurances. Recently, many researchers have reconsidered the  LQR problem from the viewpoint of   optimization  by connecting the RL method with classical convex optimization,  for example,  according to Lagrangian duality theory,
a novel primal-dual Q-learning framework for LQR was established  and  a model-free  algorithm to solve the LQR problem   was designed  in \cite{lee2018primal},  from the  viewpoint of primal-dual optimization. In \cite{farjadnasab2022model},   a new model-free algorithm was proposed on the basis of the properties of optimization frameworks.

Due to the wide application of stochastic systems, many researchers have started to focus their attentions on applying RL methods to stochastic LQR problems. For linear discrete-time stochastic systems in the presence of additive noise,  
   \cite {pang2021robust} proposed an optimistic least squares policy iteration  algorithm.       \cite{li2022model} studied the model-free design of stochastic LQR controllers  from the perspective   of primal-dual Q-learning optimization. 
For linear discrete-time systems with multiplicative noise,      \cite{gravell2020learning}  proposed a policy gradient algorithm. For stochastic systems with both multiplicative and additive noises,    \cite{lai2023model} employed Q-learning to implement the  model-free optimal control.  Based on an unbiased estimator and an initial stabilizing controller,
   \cite{jiang2024adaptive}  developed a data-driven value iteration algorithm using online data. In addition, many effective model-free RL algorithms have been proposed and convergence analyzed for stochastic optimal control problems associated with various dynamical systems, see   \cite{pang2022reinforcement, zhang2025model, jiang2024onlin} for  details.

Inspired by the above studies, this paper investigates the stochastic LQR problem for a class of discrete-time stochastic systems subject to both multiplicative and additive noises, and  a novel model-free SDP method for solving
the stochastic LQR problem  is  proposed. Compared with the existing work,  the main contributions of this paper can be summarized as follows.

1) By using the recursive relationship satisfied by the vectorized covariance matrix of the augmented states,
the stochastic LQR problem is  equivalently formulated as a nonconvex optimization problem. Instead of proving strong duality, we start directly from the dual problem of the stochastic LQR problem, and then  find out the relationship between the optimal point of the dual problem and the parameters of the Q-function by using   the convexity of the dual problem and the Karush-Kuhn-Tucker (KKT) conditions.

2)  By introducing auxiliary matrices and utilizing the technique    of the matrix direct sum, the constraint of an optimization problem containing terms related to multiplicative noise is represented in the form of linear matrix inequality  (LMI) constraint required for standard SDP. This cannot be achieved solely through the Schur complement property.

3) Based on Monte-Carlo method, a model-free implementation of the stochastic LQR controller design is given. A robust analysis of the model-free algorithm is presented by transforming errors into linear system transfer problems via dual variables. By integrating closed-loop matrix minimum singular value, the analysis unifies error boundedness and system stability, explicitly quantifying the relationship between key factors (such as sample complexity and data continuity) and errors.

The proposed model-free  SDP algorithm has the following characteristics that are worth noting:
\begin{itemize}
\item  It eliminates the need for an initial stabilizing controller -- a common limitation in existing RL methods.
\item  It does not require hyper-parameter tuning.
\item  It does not require the multiplicative and additive noises to be measurable.
\item   The algorithm implementation procedure is done in a single step, but not in an iterative form.
\item   It only needs to collect input and state information over a short length of time. 	
\end{itemize} 

Notations:  $\mathcal{R}$ is  the set of real numbers and Euclidean spaces; $\mathcal{Z}_{+} $ is the set of nonnegative integers;  $\mathcal{R}^{n}$ is the set of $n$-dimensional vectors; $\mathcal{R}^{n \times m}$ is the set of $n \times m$ real matrices. $\delta_{ij}$ is the Kronecker delta function, i.e., $\delta_{ij}= 1$  when $i = j$ and $\delta_{ij}= 0$ when $i\ne  j$. $I_{n}$ denotes the identity matrix in $\mathcal{R}^{n \times n}$,  $0$ denotes the zero vector or matrix with the appropriate dimension. $\otimes$ denotes the Kronecker product.  $\|\cdot\| _{F}$ denotes the Frobenius norm; $\|\cdot\| _{2}$ denotes the Euclidean
norm for vectors and the spectral norm for matrices. $\operatorname{col}\left(x, y\right)$ denotes the column vector consisting of vectors $x$ and $y$.   The transpose of a matrix or vector $M$ is represented by $M^{\top}$; the trace of a square matrix $M$ is denoted as $\operatorname{Tr}\left( M\right)$; spectral radius of matrix $M$ is denoted as $\rho\left( M\right)$; the minimum singular value of matrix $M$ is denoted as $\sigma_{\text {min}}\left(M \right)$; the maximum and minimum eigenvalues of a real symmetric  matrix $M$ are denoted by $\lambda_{\text {max}}\left(M \right)$ and $\lambda_{\text {min}}\left(M \right)$, respectively.  For matrix $M \in \mathcal{R}^{m \times n}$, define $\operatorname{vec}(M)=\operatorname{col}\left(M_{1}, M_{2}, \cdots, M_{n}\right)$, where $M_{j}$ is the $j$th column of matrix $M$.  $\operatorname{vec}^{-1}(\cdot) $ is an  operation such that $M = \operatorname{vec}^{-1}(\operatorname{vec}(M))$. The direct sum of matrices $M_{1}$ and $ M_{2}$ is defined as $M_{1}\oplus M_{2} \triangleq\begin{bmatrix}M_{1}& 0\\
0&M_{2}
\end{bmatrix} $.  $\mathbb{R}\left(M\right)$ denotes the range space
of $M$;  $\mathbb{N}\left(M\right)$ denotes the  kernel space
of $M$. $\mathcal{S}^{n}$, $\mathcal{S}^{n}_{+}$ and $\mathcal{S}^{n}_{++}$ are the sets of all $n \times n$ symmetric, symmetric positive semidefinite and symmetric positive definite matrices, respectively; we write  $M \succeq 0$, (resp. $M \succ 0$) if $M \in \mathcal{S}^{n}_{+}$, (resp. $M \in \mathcal{S}^{n}_{++}$). For $H\in \mathcal{S}^{n}$, define $\operatorname{vecs}( H)\triangleq\left[h_{11}, \sqrt{2} h_{12}, \ldots, \sqrt{2} h_{1 n}, h_{22}, \sqrt{2} h_{23}, \ldots, \sqrt{2} h_{(n-1)n},\right.$
$\left.  h_{n n}\right]^{\top}$, where $h_{ik}$ is the $(i, k)$th element of matrix $H$. 
\section{Problem formulation and preliminaries}
Consider the following
general linear discrete-time stochastic system with both multiplicative and additive noises  
\begin{equation}\label{eq1}
x_{k+1}=A x_{k}+B u_{k}+\left( A_{1} x_{k}+B_{1} u_{k}\right) v_{k}+w_{k},
\end{equation}
where  $k \in \mathcal{Z}_{+}$ is the discrete-time index,  $x_{k} \in \mathcal{R}^{n}$,  $u_{k} \in \mathcal{R}^{m}$, $v_{k} \in \mathcal{R}^{1}$ and $w_{k} \in \mathcal{R}^{n}$  are  the system state, control input, system multiplicative noise and  system additive noise,  respectively. $A, A_{1} \in \mathcal{R}^{n \times n}$ and $B, B_{1} \in \mathcal{R}^{n \times m}$ are the system coefficient matrices.
The
system noise sequence $\left\{\left( w_{k},v_{k}\right)\,, k \in \mathcal{Z}_{+} \right\}$ is
defined on a
complete probability space $\left(\Omega, \mathscr{F},\mathcal{P}\right)$.  The initial state $x_{0}$ is assumed to be a random variable  with   mean vector  $\mu_{0}$ and positive definite covariance matrix $\Sigma_{0}$. For convenience, it
is further assumed that the following Assumption 1 holds.
\begin{Assumption}\rm  \label{A1}  
	Assume that
	\begin{itemize} 
		\item[(\romannumeral1)] $\left\{\left( w_{k},v_{k}\right)\,, k \in \mathcal{Z}_{+} \right\}$ is independent of $x_{0}$;	 
		\item[(\romannumeral2)]  The sequence of random vectors  $\left\{w_{k}, k \in \mathcal{Z}_{+} \right\}$ is independent and identically distributed (i.i.d.) with mean zero and  covariance matrix 	$\Sigma$, where $\Sigma $ is bounded and positive definite;	
		\item[(\romannumeral3)]  The sequence of random variables  $\left\{v_{k}, k \in \mathcal{Z}_{+} \right\}$ is i.i.d. with mean zero and  variance $\sigma$,   where $0<\sigma <\bar{\sigma}$;
		\item[(\romannumeral4)]	$\left\{w_{k}, k \in \mathcal{Z}_{+} \right\}$ and $\left\{v_{k}, k \in \mathcal{Z}_{+} \right\}$ are mutually independent.
	\end{itemize}	 
\end{Assumption}

\begin{definition}\rm \cite{kubrusly1985mean}
	System~\eqref{eq1} with $u_{k} \equiv 0$ is
	called mean square stable (MSS) if  for
	any initial state $x_{0}$ and system noise sequence $\left\{\left( w_{k},v_{k}\right) , k \in \mathcal{Z}_{+} \right\}$
	satisfying Assumption~\ref{A1}, there exist $ \mu \in \mathcal{R}^{n}$ and $X \in \mathcal{R}^{n \times n}$ which are independent of $x_{0}$,
	such that 
	\begin{itemize} 
		\item[(\romannumeral1)]$\lim _{\substack{{k \rightarrow \infty}}} \|\mathbb{E}\left( x_{k}\right) -\mu\| _{2}=0\,,$
		\item[(\romannumeral2)]$\lim _{\substack{{k \rightarrow \infty}}} \|\mathbb{E}\left( x_{k}x_{k}^{\top}\right) -X\| _{2}=0\,.$
	\end{itemize}	
\end{definition}

\begin{definition}\rm  
	System~\eqref{eq1} with the feedback control policy $u_{k}=Lx_{k}$ is called    stabilizable, if the closed-loop system $	x_{k+1}=\left(A+BL\right) x_{k}+\left(A_{1}+B_{1}L\right) x_{k}v_{k}+w_{k}$ is MSS.  In this case,  the feedback control policy $u_{k}=Lx_{k}$ is called to be admissible, $L \in \mathcal{R}^{m \times n} $ is called a (mean square) stabilization gain of system~\eqref{eq1}, and the corresponding set of all (mean square) stabilizing state feedback gains is denoted as $\mathcal{L}$.  		
\end{definition}

Consequently, the cost functional can be
defined as 
\begin{equation}\label{j1}
J(L,x_{0})\triangleq\sum_{k=0}^{\infty}\alpha ^{k} \mathbb{E}\left[\begin{array}{l}
x_{k} \\
Lx_{k}
\end{array}\right]^{\top}W\left[\begin{array}{l}
x_{k} \\
Lx_{k}
\end{array}\right]\,,
\end{equation}
where $\alpha  \in (0,1) $ is the discount factor and $W \triangleq Q\oplus R $ is a block-diagonal matrix   
including the state and input weighting matrices $Q\in \mathcal{S}^{n}_{+}$ and
$R \in \mathcal{S}^{m}_{++}$, respectively. 

In this paper, we consider the infinite-horizon stochastic LQR problem. 
\begin{Problem}\rm ( Stochastic LQR Problem )\label{p1}\par	
	Find an optimal feedback gain $L^{*} \in \mathcal{L}$, if it exists,  that  minimizes the cost functional \eqref{j1}. That is, 
	\begin{equation*}
	L^{*} \triangleq
	\underset{L \in \mathcal{L}}{\arg \min }\ J(L,x_{0})\,.
	\end{equation*}  		
\end{Problem} 

The following assumptions are necessary to ensure that an optimal feedback gain $L^{*}$ exists.
\begin{Assumption}\rm \label{A2}
Assume that
	\begin{itemize} 
		\item[(\romannumeral1)] System~\eqref{eq1} is  (mean square) stabilizable. 
		\item[(\romannumeral2)]  $(A, A_{1}|\sqrt{Q})$ 
		is exactly detectable   \cite{zhang2017stochastic}.	
	\end{itemize}
\end{Assumption}
Under Assumption~\ref{A2},    the optimal value of $\underset{L \in \mathcal{L}}{ \inf }\ J(L,x_{0})$ exists and is attained, and the corresponding optimal cost $J\left(L^{*},x_{0}\right) $, which is  abbreviated as $ J^{*}$, is given by \cite{lai2023model} 	
\begin{equation*}
J^{*}=\mathbb{E}\left( x_{0}^{\top} P^{*} x_{0} \right) +\frac{\alpha}{1-\alpha}\operatorname{Tr}\left(P^{*} \Sigma \right)\,,
\end{equation*}
where $P^{*}\succeq 0$ is the unique solution to the following discrete-time generalized  algebraic Riccati equation (DGARE):
\begin{equation}\label{are1}
P= \mathscr{R}(P)\,,
\end{equation}  
where 
\begin{equation*}
\begin{aligned}
&\mathscr{R}(P)\\
\triangleq &Q+\alpha  A^{\top} PA+\alpha \sigma A_{1}^{\top} PA_{1}-\alpha ^{2}	\left( A^{\top} P B+\sigma A_{1}^{\top} P B_{1}\right) \notag \\ 
&\left( 
R+\alpha  B^{\top} PB+\alpha \sigma B_{1}^{\top} PB_{1}\right) ^{-1}	\left( B^{\top} P A+\sigma B_{1}^{\top} P A_{1}\right)\,.\notag
\end{aligned} 
\end{equation*}
The  corresponding optimal control gain $L^{*}$ for the stochastic LQR problem is given by
\begin{equation*}
\begin{aligned}
L^{*}=&-\alpha  \left( 
R+\alpha  B^{\top} P^{*}B+\alpha \sigma B_{1}^{\top} P^{*}B_{1}\right) ^{-1}\\
&\times	\left( B^{\top} P^{*} A+\sigma B_{1}^{\top} P^{*} A_{1}\right)\,.
\end{aligned}
\end{equation*}

It is clear from the above that the stochastic LQR
problem can be solved by using the knowledge of the system
dynamics. The following Q-learning offers a model-free solution for solving the stochastic LQR
problem. The Q-function is defined as (see \cite{lai2023model} for more details),
\begin{equation*}
Q\left(x_{k},u_{k}\right) \triangleq \mathbb{E}\left(U(x_{k}, u_{k})\right) +\alpha		V\left(L,x_{k+1}\right)\,,
\end{equation*}
where the value function
\begin{equation*}
V\left(L,x_{k}\right) \triangleq \mathbb{E}\sum_{i=k}^{\infty}\alpha^{i-k} U (x_{i}, Lx_{i}),
\end{equation*}
with $	U(x_{i}, u_{i})\triangleq x_{i}^{\top} Q  x_{i}+u_{i}^{\top} R  u_{i}\,.$

Denote $V^{*}(x_{k})$ as the optimal value function
generated by the optimal admissible control policy. That is,
\begin{equation*}
V^{*}\left(x_{k}\right)\triangleq V\left(L^{*},x_{k}\right)=\underset{L\in \mathcal{L} }{\operatorname{Minimize}} \  V\left(L,x_{k}\right)\,.
\end{equation*}
For the case of stochastic LQR, we have
\begin{lem}\rm \label{lem3} \cite{lai2023model}
	The optimal
	Q-function 	$Q^{*}\left(x_{k},u_{k}\right) \triangleq \mathbb{E}\left(U(x_{k}, u_{k})\right) +\alpha	V^{*}\left(x_{k+1}\right) $ can be expressed as 
	\begin{align}\label{H}
	Q^{*}\left(x_{k},u_{k}\right)= &\mathbb{E}\left[  \left[ x_{k}^{\top},u_{k}^{\top}\right] H^{*}\left[ x_{k}^{\top},u_{k}^{\top}\right]^{\top}\right] \\
	& +\frac{\alpha}{1-\alpha}\operatorname{Tr}\left(P^{*} \Sigma \right)\,,\notag
	\end{align}
	where $P^{*}\succeq 0$ is the unique solution to DGARE~\eqref{are1}, 
	$	H^{*}=\begin{bmatrix}
	H^{*}_{11}&H^{*}_{12}\\
	H^{*}_{21}&H^{*}_{22}
	\end{bmatrix}$
		with $H^{*}_{11}=Q+\alpha  A^{\top} P^{*}A +\alpha \sigma A_{1}^{\top} P^{*}A_{1}\,, H^{*}_{12}=\alpha  A^{\top} P^{*} B+\alpha \sigma A_{1}^{\top} P^{*}B_{1}=\left( H^{*}_{21}\right) ^{\top}$,  $H^{*}_{22}=R+\alpha  B^{\top} P^{*} B+\alpha \sigma B_{1}^{\top} P^{*}B_{1}$.
	
	Furthermore, the optimal control policy is presented by
	\begin{equation*}
	u^{*}_{k}  =L^{*}x_{k}=\underset{u_{k}}{\arg \min }\ Q^{*}(x_{k}, u_{k}) 
	\end{equation*}
	with the optimal state feedback gain
	\begin{equation} \label{lstar}
	L^{*}=-\left( H^{*}_{22}\right) ^{-1} \left( H^{*}_{12}\right) ^{\top}\,.	
	\end{equation}
\end{lem}

The following lemmas give  necessary and sufficient conditions for the feedback control policy to be admissible, which will be
used extensively in this paper.
\begin{lem}\rm \label{lem1}\cite{lai2023model} 
	
	The following statements are equivalent:
	\begin{itemize} 
		\item[(\romannumeral1)] The feedback control policy $u_{k}=Lx_{k}$ is  admissible for system~\eqref{eq1};
		\item[(\romannumeral2)] $\rho\left( C_{L}\right)<1 $, where $C_{L}\triangleq \left( A+BL\right)  \otimes \left( A+BL\right)+\sigma \left( A_{1}+B_{1}L\right) \otimes \left( A_{1}+B_{1}L\right)$;
		\item[(\romannumeral3)]  for each given $Y \in \mathcal{S}^{n}_{++}$, there exists a unique
		$G \in \mathcal{S}^{n}_{++}$, such that
		$\left( A+BL\right)^{\top}G\left( A+BL\right)+\sigma\left( A_{1}+B_{1}L\right)^{\top}G\left( A_{1}+B_{1}L\right)+Y=G$;
		
		\item[(\romannumeral4)] $\rho\left( \overline{C_{L}}\right)<1 $, where $\overline{C_{L}}\triangleq \overline{A_{L}} \otimes \overline{A_{L}}+\sigma \overline{A^1_{L}} \otimes \overline{A^1_{L}}$ with $\overline{A_{L}}\triangleq \left[\begin{array}{cc}A & B \\ L A & L B\end{array}\right] $ and $\overline{A^1_{L}} \triangleq \left[\begin{array}{cc}A_{1} & B_{1} \\ L A_{1} & L B_{1}\end{array}\right]$;
		\item[(\romannumeral5)] for each given $Z\in \mathcal{S}^{n+m}_{++}$, there exists a unique
		$S\in \mathcal{S}^{n+m}_{++}$, such that	$ \overline{A_{L}} S \overline{A_{L}}^{\top}+\sigma\overline{A^1_{L}} S \overline{A^1_{L}} ^{\top} +Z=S$.		  
	\end{itemize}	
\end{lem}
\begin{lem}\rm \label{lem2}\cite{li2022model}	
	Let $M\in \mathcal{R}^{n \times n}$. Then, when  $\alpha >1-\frac{\lambda_{\min }(X)}{\lambda_{\max }\left( M^{\top}YM\right) }$,   $\rho\left(M\right)<1$ if and only if for each given $X \in \mathcal{S}^{n}_{++}$, there exists a unique
	$ Y\in \mathcal{S}^{n}_{++}$, such that $\alpha M^{\top}YM+X=Y$.	
\end{lem}

\section{ Model-based  stochastic LQR  via SDP} 
The stochastic LQR problem and the Lagrange dual problem associated with the stochastic LQR problem  are first  formulated as a nonconvex optimization problem and an SDP problem, respectively. Then using the convexity of the dual problem and the KKT conditions, the relationship between the optimal point of the dual problem and the parameters of the Q-function is found. Finally, the above relationship is employed to present a novel non-iterative model-based SDP algorithm for estimating the optimal control gain.  
\subsection{Problem reformulation} 
By introducing the augmented system, which is described by 
\begin{equation*}
z_{k+1}=\overline{A_{L}} z_{k}+\overline{A^1_{L}} z_{k}v_{k}+\overline{L} w_{k}\,,
\end{equation*}
where the augmented state vector is defined as $z_{k} \triangleq \operatorname{col}\left( x_{k},u_{k}\right) $, with the initial augmented state vector $z_{0} \triangleq \operatorname{col}\left(x_{0},Lx_{0}\right) $, 
and $\overline{L} \triangleq \left[\begin{array}{c}I_{n} \\L\end{array}\right] \in \mathcal{R}^{(n+m) \times n} $,	 
we can derive an equivalent optimization reformulation of Problem~\ref{p1}  (stochastic LQR problem)  as follows.
\begin{Problem}\rm \label{p2} (Primal Problem)
	Nonconvex  optimization with optimization variables $L\in \mathcal{R}^{m \times n}$ and $S \in \mathcal{S}^{n+m}.$ 
	\begin{align}
	J_p \triangleq  \ & \underset{L\in \mathcal{R}^{m \times n},\ S \in \mathcal{S}^{n+m} }{\operatorname{Minimize}} \ \operatorname{Tr}(W S)\,, \label{eqp3}\\
	&\text { subject to } \  S \succ 0\,,\label{eqp31}\\
	&\alpha  \left( \overline{A_{L}} S \overline{A_{L}}^{\top}+\sigma\overline{A^1_{L}} S \overline{A^1_{L}} ^{\top}\right) +\overline{L} X_{0} \overline{L}^{\top} \notag \\
	&+\frac{\alpha}{1-\alpha} \overline{L}\Sigma \overline{L}^{\top}=S \,, \label{eqp32}	
	\end{align}	
	 where  $ X_{0} \triangleq \mathbb{E}\left[x_{0} x_{0}^{\top}\right]$.
\end{Problem}

Note that Problem~\ref{p2} is nonconvex since  the equality constraint \eqref{eqp32} is not linear. 
\begin{Proposition}\rm
	Problem~\ref{p2} is
	equivalent to Problem~\ref{p1} in the sense that $J_p =J^{*}$ and $L_p=L^{*}$,  where  $(S_p, L_p)$ is the optimal point of Problem~\ref{p2}. In addition,	 $(S_p, L_p)$ is unique.	
\end{Proposition}
\begin{proof}
	Using the properties of matrix trace, we can rewrite the objective function of Problem~\ref{p1}   as
	\begin{equation*}
	J\left(L, x_{0}\right)=	 \operatorname{Tr}(W S)\,, 
	\end{equation*}
	where 
	\begin{equation*}
	S \triangleq  \sum_{k=0}^{\infty}\alpha^{k} \mathbb{E}\left[\begin{array}{l}
	x_{k} \\
	L	x_{k}
	\end{array}\right]\left[\begin{array}{l}
	x_{k} \\
	L	x_{k}
	\end{array}\right]^{\top}\,.	
	\end{equation*} 
	Using  Assumption~\ref{A1},  we can obtain the following recursive relationship equation
	\begin{equation*}
	\operatorname{vec}\left( \mathbb{E}\left[z_{k+1} z_{k+1}^{\top}\right] \right) 
	= \overline{C_{L}} \operatorname{vec}\left( \mathbb{E}\left[z_{k} z_{k}^{\top}\right]\right) + \operatorname{vec} \left(  \overline{L} \Sigma \overline{L}^{\top}\right).	
	\end{equation*}	
	Thus, for $k \geq 1$,  we have  
	\begin{equation*}
	\begin{aligned}
	&\operatorname{vec}\left( \mathbb{E}\left[z_{k} z_{k}^{\top}\right]\right) \\
	=&\overline{C_{L}}^k \operatorname{vec}\left( \overline{L} X_{0}\overline{L}^{\top}\right) +\sum_{j=0}^{k-1} \overline{C_{L}}^{j} \operatorname{vec}\left( \overline{L} \Sigma \overline{L}^{\top}\right)\,,	
	\end{aligned}	
	\end{equation*}	
	where $\overline{C_{L}}\triangleq \overline{A_{L}} \otimes \overline{A_{L}}+\sigma \overline{A^1_{L}} \otimes \overline{A^1_{L}}$. 
	Hence, $\operatorname{vec}\left(S\right) $  becomes
	\begin{equation*}		
	\begin{aligned}
	&\operatorname{vec}\left(S\right)\\
	=&\sum_{k=0}^{\infty}\alpha^{k} \overline{C_{L}}^k \operatorname{vec}\left( \overline{L} X_{0}\overline{L}^{\top}\right)\\
	+ &\sum_{k=1}^{\infty}\alpha^{k}\sum_{j=0}^{k-1} \overline{C_{L}}^{j} \operatorname{vec}\left( \overline{L} \Sigma \overline{L}^{\top}\right)\,.
	\end{aligned}			
	\end{equation*}
	By an algebraic manipulation, one can obtain that
	\begin{equation*}
	\begin{aligned}
	&\operatorname{vec}\left(S\right) -\alpha\overline{C_{L}} \operatorname{vec}\left(S\right)\\ 
	=&\operatorname{vec}\left(\overline{L} X_{0} \overline{L}^{\top}\right) +\frac{\alpha}{1-\alpha} \operatorname{vec}\left( \overline{L} \Sigma \overline{L}^{\top}\right)\,,
	\end{aligned}	
	\end{equation*}
	which implies that  $S$  satisfies \eqref{eqp32}.
	
	Since $X_{0}\triangleq\mathbb{E}\left[x_{0} x_{0}^{\top}\right]=\Sigma_{0}+\mu_{0}\mu_{0}^{\top}\succ 0  $ leads to $\overline{L} \left( X_{0} +\frac{\alpha}{1-\alpha} \Sigma \right) \overline{L}^{\top}\succ 0  $. Then, when  $\alpha >1-\frac{\lambda_{\min }\left(\overline{L} \left( X_{0} +\frac{\alpha}{1-\alpha} \Sigma \right) \overline{L}^{\top}\right)}{\lambda_{\max }\left( \overline{A_{L}} S \overline{A_{L}}^{\top}+\sigma\overline{A^1_{L}} S \overline{A^1_{L}} ^{\top}\right)}$,  one can obtain from  Lemma 5 in \cite{lai2023model}, Lemma~\ref{lem1} and Lemma~\ref{lem2}  that  $\rho\left(\overline{C_{L}}\right)<1$ if and only if \eqref{eqp32} has a unique  solution $S \succ 0$. It follows from  Lemma~\ref{lem1}    that  
	$L \in \mathcal{L}$. Therefore, we can replace the constraint  $L \in \mathcal{L}$  in Problem~1 by  $S \succ 0$  without changing its optimal point.
	
	If  $\left(S_{p}, L_{p}\right)$  is the optimal point of Problem~\ref{p2} and the corresponding optimal value is  $J_{p}$,  then  $ S_{p} \succ 0$  and  $\rho\left(\overline{C_{L_{p}}}\right)<1$. Thus $ L_{p} \in \mathcal{L}$  is a feasible point of Problem~\ref{p1}, thereby,  $J_{p} \geq J^{*}$. Furthermore,  if  $S_{p} $ is the unique solution of \eqref{eqp32} with  $L_{p}=L^{*} \in \mathcal{L}$,  then the resulting objective function of  Problem~\ref{p2} is  $J_{p}=J^{*}$. Thus, we can conclude that    $J_{p}=J^{*}$. The uniqueness of  $\left(S_{p}, L_{p}\right)$  can be shown by the uniqueness of  $L^{*}$.\qed
\end{proof}
\begin{Remark}\rm
	Unlike the system studied in \cite{li2022model}, which contains only additive noise, the system in this paper is more general and contains both multiplicative and additive noises. The presence of 
	multiplicative noise makes the generalized expression of the augmented system more complex and difficult to obtain.  To address this challenge, unlike the approach in \cite{li2022model} which establishes equations satisfied by $S$ using the derived general term formula for augmented system $z_{k+1}$, we  directly focus on the term $\mathbb{E}\left[z_{k+1} z_{k+1}^{\top}\right]$, utilize the technique of vectorization to obtain the constraint satisfied by $\operatorname{vec}\left(S\right)$ directly from the recursive relationship satisfied by $\operatorname{vec}\left( \mathbb{E}\left[z_{k+1} z_{k+1}^{\top}\right] \right) $, and then use the properties of the Kronecker product and $S = \operatorname{vec}^{-1}(\operatorname{vec}(S))$ to obtain the constraint satisfied by $S$.
\end{Remark}
\subsection{Model-based  stochastic LQR  via SDP} 
We begin by representing   the  Lagrange dual problem associated with the nonconvex optimization in Problem~\ref{p2} as a standard convex optimization problem as shown in Problem~\ref{p3}.
\begin{Problem}\rm	\label{p3}  (Dual Problem)
	Convex  optimization with optimization variables $M \in \mathcal{R}^{n \times n}$ and $F =\begin{bmatrix}
	F_{11}&F_{12}\\
	F_{12}^{\top}&F_{22}
	\end{bmatrix} \in \mathcal{S}^{n+m}$ with $F_{11}\in \mathcal{S}^{n}$, $F_{22}\in \mathcal{S}^{m}$,  and  $F_{12} \in \mathcal{R}^{n \times m}$. 	
	\begin{align}
	&\underset{F,M }{\operatorname{Maximize}}  \ \operatorname{Tr}\left( \left( X_{0}+\frac{\alpha}{1-\alpha} \Sigma\right) M\right), \label{dp1}\\
	&	\text { subject to } \ F_{22}\succ 0\,, \label{dp1-3} \\ 		
	& \mathscr{P}(F)-M\succeq 0\,,\label{dp1-1}\\			
	& \alpha  [\begin{array}{ll}
	A & B
	\end{array}]^{\top}\mathscr{P}(F)[\begin{array}{ll}
	A & B
	\end{array}] -F+W\notag \\
	&+\alpha \sigma [\begin{array}{ll}
	A_{1} & B_{1}
	\end{array}]^{\top}\mathscr{P}(F)[\begin{array}{ll}
	A_{1} & B_{1}
	\end{array}]  \succeq 0\,, \label{dp1-2}
	\end{align} 
	where $\mathscr{P}(F)\triangleq F_{11}-F_{12} F_{22}^{-1} F_{12}^{\top}$.
\end{Problem}
\begin{theorem}\rm  
	Problem~\ref{p3} is an equivalent constrained convex optimization formulation of the Lagrange  dual problem  associated with   Problem~\ref{p2}.	
\end{theorem}

\begin{proof} 
	The Lagrangian  associated with Problem~\ref{p2} is	
	\begin{equation*}
	\begin{aligned}
	&\mathscr{L}_{1}(S,L, F_{0}, F)\\
	\triangleq &\operatorname{Tr}(W S)-\operatorname{Tr}(F_{0} S) \\
	&+\operatorname{Tr}\left(\left(\alpha  \left( \overline{A_{L}} S \overline{A_{L}}^{\top}+\sigma \overline{A^1_{L}} S \overline{A^1_{L}}^ {\top}\right)+\overline{L} X_{0}\overline{L}^{\top}\right.\right. \\
	& \left.\left.+\frac{\alpha}{1-\alpha} \overline{L} \Sigma \overline{L}^{\top}-S\right) F\right) \\
	=&\operatorname{Tr}\left(\left(\alpha  \overline{A_{L}}^{\top} F \overline{A_{L}}+\alpha \sigma \overline{A^1_{L}}^ {\top} F\overline{A^1_{L}}-F-F_{0}+W\right) S\right) \\
	&+\operatorname{Tr}\left(\left( X_{0}+\frac{\alpha}{1-\alpha} \Sigma\right) \overline{L} ^{\top} F\overline{L}\right)\,,	
	\end{aligned}			
	\end{equation*}
	where the matrices $F_{0}\in \mathcal{S}^{n+m}_+$ and $F\in \mathcal{S}^{n+m}$ are the Lagrange multipliers associated with the inequality and equality constraints of Problem~\ref{p2}, respectively. 
	
The dual problem is $\underset{F_{0}, F }{\operatorname{Maximize}}\ \mathcal{D}_{1}(F_{0}, F)$,
where the  Lagrange dual function $\mathcal{D}_{1}(F_{0}, F)$   is defined as	
	\begin{equation*}
	\mathcal{D}_{1}(F_{0}, F)\triangleq \underset{L} {\inf} \ \underset{S\succeq 0} {\inf} \ \mathscr{L}_{1}(S,L,F_{0}, F)\,.	
	\end{equation*}
	
	When the Lagrangian $ \mathscr{L}_{1}(S,L,F_{0}, F)$ is unbounded  from below in $S$ and $L$, the dual function $\mathcal{D}_{1}(F_{0}, F)$ takes on the
	value $-\infty$. Thus for the Lagrange dual problem: 
	\begin{equation*}
	\underset{F_{0}, F }{\operatorname{Maximize}} \   \mathcal{D}_{1}(F_{0}, F)\,,
	\end{equation*}
	the dual feasibility requires that $F_{0} \succeq 0$ and that $\mathcal{D}_{1}(F_{0}, F)>-\infty$.  For given $F_{0}$ and $F$, since $S\succ 0$, the Lagrange dual function  $\mathcal{D}_{1}(F_{0}, F)$  becomes
	\begin{equation*}
	\begin{array}{l}
	\mathcal{D}_{1}(F_{0}, F)\\
	=  \left\{\begin{array}{ll}
	\underset{L  }{\inf} \ \operatorname{Tr}\left(\left( X_{0}+\frac{\alpha}{1-\alpha} \Sigma\right) \overline{L} ^{\top} F\overline{L}\right)\,, & \text { if } (F_{0}, F) \in \mathcal{F}
	\\
	-\infty\,, & \text { otherwise }
	\end{array}\right.
	\end{array}	
	\end{equation*}
	where 
	\begin{equation*}
	\begin{aligned}
	&\mathcal{F}\triangleq\left\lbrace (F_{0}, F):\right.\\&\left.\alpha  \overline{A_{L}}^{\top} F \overline{A_{L}}+\alpha \sigma \overline{A^1_{L}}^ {\top} F\overline{A^1_{L}}-F-F_{0}+W \succeq0,\forall  L\in \mathcal{L}\right\rbrace. \end{aligned}
	\end{equation*} 
	Using the block-diagonal matrix representation of the Lagrange multiplier $F$, 
	\begin{equation*}
	\mathscr{L}_{2}(L, F) \triangleq \operatorname{Tr}\left(\left( X_{0}+\frac{\alpha}{1-\alpha} \Sigma\right) \overline{L} ^{\top} F\overline{L}\right)
	\end{equation*}  
	can be represented as
	\begin{equation*}
	\begin{aligned}
	&\mathscr{L}_{2}(L, F)=\operatorname{Tr} \left(\left( X_{0}+\frac{\alpha}{1-\alpha} \Sigma\right)\right.\\
	&\left. \left(L^{\top} F_{22} L+L^{\top} F_{12}^{\top}+F_{12} L+F_{11}\right)\right)\,. 			
	\end{aligned}		 	
	\end{equation*}	
	Since the Lagrange multiplier $F_{0}$ only appears in the constraint condition: 
	\begin{equation*}
	\alpha  \overline{A_{L}}^{\top} F \overline{A_{L}}+\alpha \sigma \overline{A^1_{L}}^ {\top} F\overline{A^1_{L}}-F-F_{0}+W \succeq 0, 
	\end{equation*}
	the dual feasibility is equivalent to	
	\begin{equation*}
	\alpha  \overline{A_{L}}^{\top} F \overline{A_{L}}+\alpha \sigma \overline{A^1_{L}}^ {\top} F\overline{A^1_{L}}-F+W \succeq 0. 
	\end{equation*}
	Therefore, the Lagrange  dual problem associated with Problem~\ref{p2}  
	is equivalent to
	\begin{align}
	&\underset{ F}{\operatorname{Maximize}} \  \underset{ L } {\inf} \  \mathscr{L}_{2}(L, F), \label{dp2-1}\\
	&\text { subject to }   	\alpha  \overline{A_{L}}^{\top} F \overline{A_{L}}+\alpha \sigma \overline{A^1_{L}}^ {\top} F\overline{A^1_{L}}-F+W \succeq 0 \,. \label{dp2-2}
	\end{align}		
	In order to   obtain an explicit expression for the function  $\mathcal{D}_2 ( F)\triangleq\underset{ L } {\inf} \  \mathscr{L}_{2}(L, F)$, consider the
	following three different    cases: 
	
	Case 1: Let $ F_{22}\succ 0$. In this case, let $\frac{\partial \mathscr{L}_{2}(L, F)}{\partial L}=0$, one can obtain
	\begin{equation*}
	\begin{array}{l}
	F_{22} L \left( X_{0}+\frac{\alpha}{1-\alpha} \Sigma\right) +F_{22}^{\top} L \left( X_{0}+\frac{\alpha}{1-\alpha} \Sigma\right) ^{\top}\\
	+F_{12}^{\top} \left( X_{0}+\frac{\alpha}{1-\alpha} \Sigma\right) +F_{12}^{\top} \left( X_{0}+\frac{\alpha}{1-\alpha} \Sigma\right) ^{\top}=0\,. 
	\end{array}		
	\end{equation*}	
	By solving the above equation, the optimal  point of the function $\mathscr{L}_{2}(L, F)$  can be solved to be $L^{*}=-F_{22}^{-1} F_{12}^{\top}$, so the  infimum of $\mathscr{L}_{2}(L, F)$ is attained at $L^{*}$,  and hence 
	\begin{equation}\label{oj2}
	\begin{aligned}
	\mathcal{D}_2 ( F) 
	=&   \mathscr{L}_{2}(L^{*}, F)\\
	=&\operatorname{Tr} \left(\left( X_{0}+\frac{\alpha}{1-\alpha} \Sigma\right) \mathscr{P}(F)\right)\,.
	\end{aligned}	
	\end{equation}
	Substituting $L^{*}$ into the constraint \eqref{dp2-2}  and combining it with the objective function in \eqref{oj2} yields an alternative equivalent representation of the dual problem  \eqref{dp2-1}--\eqref{dp2-2}:
	\begin{align*}
	\underset{F }{\operatorname{Maximize}} &\ \operatorname{Tr} \left(\left( X_{0}+\frac{\alpha}{1-\alpha} \Sigma\right) \mathscr{P}(F)\right)\,,\\
	\text { subject to } \  & 	\alpha  [\begin{array}{ll}
	A & B
	\end{array}]^{\top}\mathscr{P}(F)[\begin{array}{ll}
	A & B
	\end{array}]-F+W\\
	+&\alpha \sigma [\begin{array}{ll}
	A_{1} & B_{1}
	\end{array}]^{\top}\mathscr{P}(F)[\begin{array}{ll}
	A_{1} & B_{1}
	\end{array}]
	\succeq 0\,.
	\end{align*}		
	Case 2: Let $F_{22} \succeq 0 $ be a singular matrix, suppose that $\mathbb{R}\left(F_{12}^{\top}\right) \subseteq   \mathbb{R}\left(F_{22}\right)$. In this case, the infimum of $\mathscr{L}_{2}(L, F)$ still exists. Suppose $\operatorname{rank}\left( F_{22}\right)  = q<m $, then $F_{22}$ can be factored as $F_{22}=U_{q} \Lambda_{q} U_{q}^{\top}$,	 where $ U_{q}\in \mathcal{R}^{m \times q}$ satisfies $ U_{q}^{\top}U_{q} =I_{q}$ and $\Lambda_{q} =
	\operatorname{diag}(\lambda_1,\cdots,\lambda_q)$ with $\lambda_1\geq \lambda_2 \geq\cdots\geq\lambda_q>0$.
	Moreover, the columns of  $ U_{q} $ span $\mathbb{R}\left(F_{22}\right)$. The pseudo-inverse of the singular matrix $F_{22}$ can be represented as $F_{22}^{\dagger}=U_{q} \Lambda_{q}^{-1} U_{q}^{\top}$. Following a procedure similar to the Case 1, the Lagrange  dual problem associated with Problem~\ref{p2}  is derived as
	\begin{align*}
	\underset{F }{\operatorname{Maximize}} &\ \operatorname{Tr} \left(\left( X_{0}+\frac{\alpha}{1-\alpha} \Sigma\right) \left(F_{11}-F_{12} F_{22}^{\dagger} F_{12}^{\top}\right)\right)\,,\\
	\text { subject to } \ & 	\alpha  [\begin{array}{ll}
	A & B
	\end{array}]^{\top}\left(F_{11}-F_{12} F_{22}^{\dagger} F_{12}^{\top}\right)[\begin{array}{ll}
	A & B
	\end{array}]\\
	+&\alpha \sigma [\begin{array}{ll}
	A_{1} & B_{1}
	\end{array}]^{\top}\left(F_{11}-F_{12} F_{22}^{\dagger} F_{12}^{\top}\right)[\begin{array}{ll}
	A_{1} & B_{1}
	\end{array}]\\
	-&F+W \succeq 0\,.
	\end{align*}				
	Since $\mathbb{R}\left(F_{12}^{\top}\right) \subseteq   \mathbb{R}\left(F_{22}\right)$, there exists  an unitary matrix $U_{m-q}$ such that the columns of $U_{m-q}$ span  $\mathbb{N}\left(F_{12}\right)$. Let   $\Lambda_{m-q}$ be an arbitrary diagonal matrix with positive diagonal elements, and construct a matrix $\overset{\frown} {F}_{22}\succ 0$ as
	\begin{equation*}
	\overset{\frown} {F}_{22}=[\begin{array}{ll}
	U_{q} & U_{m-q}
	\end{array}]\left[ \begin{array}{cc}
	\Lambda_{q} & 0 \\
	0 & \Lambda_{m-q}
	\end{array}\right] [\begin{array}{ll}
	U_{q} & U_{m-q}
	\end{array}]^{\top} \,. 	
	\end{equation*}
	Noting  that
	\begin{equation*}
	\begin{array}{l}
	F_{12} \overset{\frown} {F}_{22}^{-1} F_{12}^{\top} \\
	=F_{12}[\begin{array}{ll}
	U_{q} & U_{m-q}
	\end{array}]\left[\begin{array}{cc}
	\Lambda_{q}^{-1} & 0 \\
	0 & \Lambda_{m-q}^{-1}
	\end{array}\right]\left[\begin{array}{c}
	U_{q}^{\top} \\
	U_{m-q}^{\top}
	\end{array}\right] F_{12}^{\top} \\
	=F_{12} U_{q}\Lambda_{q}^{-1} U_{q}^{\top} F_{12}^{\top}-\underbrace{F_{12} U_{m-q} \Lambda_{m-q}^{-1} U_{m-q}^{\top} F_{12}^{\top}}_{0}\\
	=F_{12} F_{22}^{\dagger} F_{12}^{\top}\,,
	\end{array}	
	\end{equation*}		
	it can be concluded that every  admissible  candidate $F_{22} \succeq 0$ leads to an objective value   which can be obtained by replacing a  counterpart 
	$\overset{\frown} {F}_{22}\succ0$.
	
	Case 3: $ F_{22}$  has at least one negative eigenvalue or $F_{22} \succeq 0$  while   $\mathbb{R}\left(F_{12}^{\top}\right) \nsubseteq \mathbb{R}\left(F_{22}\right)$. In this case, the Lagrangian is unbounded from below through minimizing $\mathscr{L}_{2}(L, F)$ in terms of $L$.

	To summarize, the existence of $\underset{ L } {\inf}  \mathscr{L}_{2}(L, F) $ requires that (i) $F_{22}\succ0$ or (ii) $F_{22} \succeq 0 $ with an additional constraint  $\mathbb{R}\left(F_{12}^{\top}\right) \subseteq \mathbb{R}\left(F_{22}\right)$. Moreover, every permissible $F_{22} \succeq 0$ results in an objective value that can be alternatively derived by a  counterpart  $\overset{\frown} {F}_{22}\succ0$. Thus, without loss of generality,  we follow up with the assumption that $F_{22}\succ 0$.  Therefore, the dual problem \eqref{dp2-1}--\eqref{dp2-2} is equivalent to
	\begin{align}
	&\underset{F }{\operatorname{Maximize}} \ \operatorname{Tr} \left(\left( X_{0}+\frac{\alpha}{1-\alpha} \Sigma\right) \mathscr{P}(F)\right)\,,\label{dp3-1}\\
	&\text { subject to } \   F_{22}\succ 0\,, \label{dp3-3}\\
	&	\alpha  [\begin{array}{ll}
	A & B
	\end{array}]^{\top}\mathscr{P}(F)[\begin{array}{ll}
	A & B
	\end{array}]-F+W   \notag  \\ 
	+&\alpha \sigma [\begin{array}{ll}
	A_{1} & B_{1}
	\end{array}]^{\top}\mathscr{P}(F)[\begin{array}{ll}
	A_{1} & B_{1}
	\end{array}] 
	\succeq 0\,. \label{dp3-2}
	\end{align}		
	Finally, by introducing the slack variable $ M \in \mathcal{S}^{n}$ and the additional constraint \eqref{dp1-1},	
	the Lagrange  dual problem	 \eqref{dp3-1}--\eqref{dp3-2} can be equivalently described by Problem~\ref{p3}.	In fact, let    $F^{*}$  be the optimal point of the dual problem \eqref{dp3-1}--\eqref{dp3-2}  and $\left(F^{*},M^{*} \right) $  be the optimal point of Problem~\ref{p3}. According to the constraint \eqref{dp1-1} and positive definiteness of   $X_{0}+\frac{\alpha}{1-\alpha} \Sigma$, 
	\begin{equation*}
	\begin{aligned}
	&\operatorname{Tr}  \left(\left( X_{0}+\frac{\alpha}{1-\alpha} \Sigma\right)  M^{*}\right) \\
	\leq &\operatorname{Tr}\left(\left( X_{0}+\frac{\alpha}{1-\alpha} \Sigma\right) \mathscr{P}(F^{*})\right)\,.	
	\end{aligned}
	\end{equation*}
	Since $M^{*}$ is the maximum point of the objective function~\eqref{dp1-1}, 
	\begin{equation*}
	\begin{aligned}
	&\operatorname{Tr}\left(\left( X_{0}+\frac{\alpha}{1-\alpha} \Sigma\right)  M^{*}\right)\\
	\geq&\operatorname{Tr}\left(\left( X_{0}+\frac{\alpha}{1-\alpha} \Sigma\right) \mathscr{P}(F^{*})\right)\,,
	\end{aligned}	
	\end{equation*}
	and hence
	\begin{equation*}
	\operatorname{Tr}  \left(\left( X_{0}+\frac{\alpha}{1-\alpha} \Sigma\right) \left(M^{*}-\mathscr{P}(F^{*})\right)\right)
	=0.
	\end{equation*} 
	Consequently, all of the eigenvalues of the matrix  
	$M^{*}-\mathscr{P}(F^{*})$
	are equal to zero. Since  $M^{*}$  is symmetric, it follows that
	\begin{equation*}
	M^{*}=\mathscr{P}(F^{*}) \,,
	\end{equation*}	
	which means that the dual problem	 \eqref{dp3-1}--\eqref{dp3-2} is equivalent to  Problem~\ref{p3}.\qed	
\end{proof}
The next theorem gives  the  relationship between the optimal point to the dual problem associated with the stochastic LQR problem  and the parameters of the optimal Q-function. 

\begin{theorem}\rm	\label{th2}
	The optimal point of Problem~\ref{p3}, which
	is denoted by $\left(F^{*},M^{*} \right) $,  is independent of  $ X_{0}+\frac{\alpha}{1-\alpha} \Sigma$. Furthermore, $F^{*}=H^{*}$ and $M^{*}=P^{*}$,  where $ H^{*}$  is the  parameter of the optimal Q-function  expressed in \eqref{H} and  $P^{*}$  is the solution to  DGARE~\eqref{are1}.
\end{theorem}
\begin{proof}
	By introducing the Lagrange multipliers $G_{1}\in \mathcal{S}^{n+m}_+$, $G_{2}\in \mathcal{S}^{m}_+$ and $G_{3}\in \mathcal{S}^{n}_+$ associated with the inequality  constraints of Problem~\ref{p3},
	the Lagrangian associated with Problem~\ref{p3} is described by
	\begin{equation*}
	\begin{array}{l}
	\mathscr{L}_{3}\left(F, M, G_{1}, G_{2}, G_{3}\right)\\
	=-\operatorname{Tr}\left( \left( X_{0}+\frac{\alpha}{1-\alpha} \Sigma\right)  M\right)  \\	
	-\operatorname{Tr}\left(G_{1} \left(\alpha  [\begin{array}{ll}
	A & B
	\end{array}]^{\top}\mathscr{P}(F)[\begin{array}{ll}
	A & B
	\end{array}]\right.\right.\\
	\left.\left.+\alpha\sigma [\begin{array}{ll}
	A_{1} & B_{1}
	\end{array}]^{\top}\mathscr{P}(F)[\begin{array}{ll}
	A_{1} & B_{1}
	\end{array}] -F+W\right) \right) \\
	-\operatorname{Tr}\left( G_{2}F_{22}\right)
	-\operatorname{Tr}\left(G_{3}\left(\mathscr{P}(F)-M\right)\right) \,.
	\end{array}
	\end{equation*}
	Since Problem~\ref{p3} is convex, the  KKT conditions are  sufficient for the
	points to be primal and dual optimal \cite{boyd2004convex}.  In other words, if $\tilde{F}$, $\tilde{M}$, $\tilde{G}_{1}$,  $\tilde{G}_{2}$ and $\tilde{G}_{3}$ are any points that satisfy the following KKT conditions of Problem~\ref{p3}:
		\begin{itemize} 
		\item[(\romannumeral1)] Primal feasibility: \eqref{dp1-3}--\eqref{dp1-2};
		\item[(\romannumeral2)] Dual feasibility:
		\begin{equation*}
		G_{1} \succeq 0, \quad G_{2} \succeq 0, \quad G_{3} \succeq 0 \,; 
		\end{equation*}	
		\item[(\romannumeral3)] Complementary slackness:
		\begin{align}
		&\operatorname{Tr}\left(G _ { 1 } \left(\alpha [\begin{array}{ll}
		A & B
		\end{array}]^{\top}\mathscr{P}(F) [\begin{array}{ll}
		A & B
		\end{array}]-F+W\right. \right.\notag \\
		&\left.\left.
		+\alpha\sigma [\begin{array}{ll}
		A_{1} & B_{1}
		\end{array}]^{\top}\mathscr{P}(F)[\begin{array}{ll}
		A_{1} & B_{1}
		\end{array}]	\right) \right) =0\,, 	\label{cs2}\\
		&\operatorname{Tr}\left(G_{2}F_{22}\right)=0\,, 	\label{cs3}\\
		&\operatorname{Tr}\left(G_{3}\left(\mathscr{P}(F)-M\right)\right)=0; \label{cs1}
		\end{align}			
		\item[(\romannumeral4)] Stationarity conditions:
		\begin{equation}\label{sc}
		\frac{\partial\mathscr{L}_{3}}{\partial M} =0\,,\quad  \frac{\partial\mathscr{L}_{3}}{\partial F} =0 \,,
		\end{equation}
	\end{itemize} 
	then  $\left(\tilde {F}, \tilde{M}\right) $ and $\left( \tilde{G}_{1},  \tilde{G}_{2}, \tilde{G}_{3}\right)$ are primal and dual   optimal points, respectively.  
	
	Choose $\tilde{F}=H^{*}$ and $\tilde{M}=P^{*}$. Substituting  $\tilde{F}=H^{*}$ into the left-hand side of \eqref{dp1-2} results in the left-hand side of~\eqref{dp1-2} to be equal to $0$, 
	from which it follows that the constraint~\eqref{dp1-2} is active, i.e., the constraint~\eqref{cs2} is trivially satisfied without any requirement for  $\tilde G_{1}$. The constraint~\eqref{dp1-3} is not active at the candidate point  $(\tilde{F}, \tilde{M})=\left( H^{*}, P^{*}\right) $ due to $\tilde{F}_{22}=H^{*}_{22}\succ 0$, which dictates   $\tilde{G}_{2}=0$. Moreover, substituting  $(\tilde{F}, \tilde{M})=\left( H^{*}, P^{*}\right) $  into the left-hand side of the inequality constraint~\eqref{dp1-1} results in the left-hand side of~\eqref{dp1-1} to be equal to $0$.	That is, the constraint~\eqref{dp1-1} is active and \eqref{cs1} is trivially satisfied without imposing any restriction on  $\tilde G_{3}$.	
	
	It is shown below that the stationary conditions \eqref{sc} can be satisfied at the  point $(\tilde{F}, \tilde{M})$  by appropriate selection of $\tilde{G}_{1}$ and $\tilde{G}_{3}$. To this end, the matrix variable $F$ is reformulated
	as 
	\begin{equation*}
	F=E_{1} F_{11} E_{1}^{\top}+E_{1} F_{12} E_{2}^{\top}+E_{2} F_{12}^{\top} E_{1}^{\top}+E_{2} F_{22} E_{2}^{\top}
	\end{equation*} 
	by using the matrices $E_{1}= [\begin{array}{ll}I_{n} & 0\end{array}]^{\top} \in \mathcal{R}^{(n+m) \times n}$ and  $E_{2}=[\begin{array}{ll}0 & I_{m}\end{array}]^{\top} \in \mathcal{R}^{(n+m) \times m} $. Then, the stationarity conditions \eqref{sc}, can be equivalently described by  \eqref{sc1}--\eqref{sc4}.
	\begin{align}
	\frac{\partial\mathscr{L}_{3}}{\partial M} =&-\left( X_{0}+\frac{\alpha}{1-\alpha} \Sigma\right) ^{\top}+G_{3}^{\top}=0\,, \label{sc1}\\
	\frac{\partial\mathscr{L}_{3}}{\partial F_{11}}  =&-G_{3}+E_{1}^{\top} G_{1} E_{1}-\alpha [\begin{array}{ll}
	A & B
	\end{array}] G_{1}[\begin{array}{ll}
	A & B
	\end{array}]^{\top}\notag\\
	&-\alpha\sigma [\begin{array}{ll}
	A_{1} & B_{1}
	\end{array}]G_{1}[\begin{array}{ll}
	A_{1} & B_{1}
	\end{array}]^{\top}=0\,, \label{sc2}\\
	\frac{\partial\mathscr{L}_{3}}{\partial F_{12}}  =&2\alpha \left( [\begin{array}{ll}
	A & B
	\end{array}]G_{1}[\begin{array}{ll}
	A & B
	\end{array}]^{\top}\right. \notag\\
	& \left.+\sigma [\begin{array}{ll}
	A_{1} & B_{1}
	\end{array}]G_{1}[\begin{array}{ll}
	A_{1} & B_{1}
	\end{array}]^{\top}\right) \tilde{F}_{12} \tilde{F}_{22}^{-1}\notag\\
	&+2G_{3} \tilde{F}_{12} \tilde{F}_{22}^{-1} +2E_{1}^{\top} G_{1} E_{2}=0\,, \label{sc3}\\
	\frac{\partial\mathscr{L}_{3}}{\partial F_{22}}  =&-\tilde{F}_{22}^{-1}F_{12}^{\top}\left(
	G_{3} 
	+\alpha [\begin{array}{ll}
	A & B
	\end{array}] G_{1}[\begin{array}{ll}
	A & B
	\end{array}]^{\top} \right. \notag\\
	&\left.+\alpha\sigma [\begin{array}{ll}
	A_{1} & B_{1}
	\end{array}]G_{1}[\begin{array}{ll}
	A_{1} & B_{1}
	\end{array}]^{\top} \right) \tilde{F}_{12}\tilde{F}_{22}^{-1}\notag\\
	&+E_{2}^{\top} G_{1} E_{2}=0\,. \label{sc4}
	\end{align}
	Choose 
	\begin{equation}\label{g3}
	\tilde{G}_{3}= X_{0}+\frac{\alpha}{1-\alpha} \Sigma \succ 0\,, 
	\end{equation}
	which satisfies the stationary condition \eqref{sc1}.
	
	From \eqref{sc2} and \eqref{sc3}, the matrix $G_{1}$ should be subject to the following constraint:
	\begin{equation}\label{eq25}
	E_{1}^{\top} G_{1} E_{1} \tilde{F}_{12} \tilde{F}_{22}^{-1}+E_{1}^{\top} G_{1} E_{2}=0\,. 
	\end{equation}
	Furthermore, from~\eqref{sc2} and \eqref{sc4}, $G_{1}$ should also be subject to the following constraint:
	\begin{equation}\label{eq26}
	-\tilde{F}_{12}^{\top} E_{1}^{\top} G_{1} E_{1} \tilde{F}_{12}+\tilde{F}_{22} E_{2}^{\top} G_{1} E_{2} \tilde{F}_{22}=0\,.
	\end{equation}
	By partitioning the Lagrange multiplier $G_{1}$ as $G_{1}=\left[\begin{array}{ll}G^{1} _{11} & G^{1}_{12} \\ \left( G^{1}_{12}\right) ^{\top} & G^{1}_{22}\end{array}\right]$, the equalities  \eqref{eq25} and \eqref{eq26} are represented as
	\begin{align}
	&G^{1} _{11} \tilde{F}_{12} \tilde{F}_{22}^{-1}+G^{1} _{12}=0\,, \label{eq27}\\
	&-\tilde{F}_{12}^{\top} G^{1} _{11} \tilde{F}_{12}+\tilde{F}_{22}G^{1} _{22}\tilde{F}_{22}=0\,.\label{eq28}
	\end{align}
	By substituting \eqref{g3},   \eqref{eq27} and \eqref{eq28} into \eqref{sc2}, we have
	\begin{align}
	&\tilde G_{3}+\alpha\left( A-B \tilde{F}_{22}^{-1} \tilde{F}_{12}^{\top}\right)G^{1}_{11}\left(A-B \tilde{F}_{22}^{-1} \tilde{F}_{12}^{\top}\right)^{\top} \notag \\
	&+\alpha\sigma\left(A_{1}-B_{1} \tilde{F}_{22}^{-1} \tilde{F}_{12}^{\top}\right)G^{1}_{11}\left(A_{1}-B_{1} \tilde{F}_{22}^{-1} \tilde{F}_{12}^{\top}\right)^{\top} \notag\\
	=&G^{1}_{11} \,.\label{eq29}
	\end{align}	
		As we choose to let $\tilde{F}=H^{*}$, based on the optimal control gain representation given by \eqref{lstar}, $A-B \tilde{F}_{22}^{-1} \tilde{F}_{12}^{\top}$ and $A_{1}-B_{1} \tilde{F}_{22}^{-1} \tilde{F}_{12}^{\top}$ in \eqref{eq29} become    $\mathcal{A}\triangleq A+BL^{*}$  and $\mathcal{A}_{1} \triangleq A_{1}+B_{1}L^{*}$,  respectively. Since $L^{*} \in \mathcal{L}$ leads to $\rho\left(C_{L^{*}}\right)<1$, then when  
	\begin{equation*}
	\alpha>1-\frac{\lambda_{\min }\left(\tilde G_{3} \right)}{\lambda_{\operatorname{max} }\left(\mathcal{A}G^{1}_{11}\mathcal{A}^{\top}+\sigma\mathcal{A}_{1}G^{1}_{11}\mathcal{A}_{1}^{\top}
		\right)}\,,	
	\end{equation*}  
	  one can obtain from  Lemma~\ref{lem1} and Lemma~\ref{lem2}   that  for  given $\tilde G_{3}= X_{0}+\frac{\alpha}{1-\alpha} \Sigma\succ0$, there exists a unique $\tilde G^{1}_{11}\succ 0$  such that \eqref{eq29} holds. 
	Accordingly,  construct $\tilde{G}_{1}$  as 
	\begin{equation*}
	\begin{aligned}
	\tilde{G}_{1} & =\left[\begin{array}{cc}
	\tilde{G}^{1}_{11} & -\tilde{G}^{1}_{11} \tilde{F}_{12} \tilde{F}_{22}^{-1} \\
	-\tilde{F}_{22}^{-1} \tilde{F}_{12}^{\top} \tilde{G}^{1}_{11} & \tilde{F}_{22}^{-1} \tilde{F}_{12}^{\top} \tilde{G}^{1}_{11} \tilde{F}_{12} \tilde{F}_{22}^{-1}
	\end{array}\right] \\
	& =\left[\begin{array}{c}
	I_{n} \\
	-\tilde{F}_{22}^{-1} \tilde{F}_{12}^{\top}
	\end{array}\right] \tilde{G}^{1}_{11}\left[\begin{array}{c}
	I_{n} \\
	-\tilde{F}_{22}^{-1} \tilde{F}_{12}^{\top}
	\end{array}\right] ^{\top}\succeq  0\,.
	\end{aligned}
	\end{equation*}	 
	In summary, it can be concluded that  $\left( H^{*},P^{*}\right)$ and 
	\begin{equation*}
	\left(\left[\begin{array}{c}
	I_{n} \\
	-\tilde{F}_{22}^{-1} \tilde{F}_{12}^{\top}
	\end{array}\right] \tilde{G}^{1}_{11}\left[\begin{array}{c}
	I_{n} \\
	-\tilde{F}_{22}^{-1} \tilde{F}_{12}^{\top}
	\end{array}\right] ^{\top}\,, 0\,,X_{0}+\frac{\alpha}{1-\alpha} \Sigma_0 \right)
	\end{equation*}
	are primal and dual optimal points, respectively.
	Therefore,  $F^{*}=H^{*}$ and $M^{*}=P^{*}$. Moreover, according to the expressions of $H^{*}$ and $P^{*}$, $\left( H^{*}, P^{*}\right) $ is  independent of $ X_{0}+\frac{\alpha}{1-\alpha} \Sigma$.\qed
\end{proof}
\begin{Remark}\rm	
	The primal problem in this paper is nonconvex, which makes it nontrivial to prove that strong duality holds. Instead of proving that strong duality holds, this paper proceeds directly from the dual problem and finds the relationship between the optimal point of the dual problem and the parameters in Q-learning by using the convexity of the dual problem and the conclusion that the KKT  conditions in convex optimization are sufficient  for optimality \cite{boyd2004convex}.
\end{Remark}

Theorem~\ref{th2} shows that  the term  $ X_{0}+\frac{\alpha}{1-\alpha} \Sigma $ only affects the optimal value of the objective function~\eqref{dp1} in Problem~\ref{p3}, and that optimal point   $\left(F^{*}, M^{*}\right) $  is independent of  $ X_{0}+\frac{\alpha}{1-\alpha} \Sigma $.
Then, without loss of generality, the objective function~\eqref{dp1} in  Problem~\ref{p3} can be expressed as
\begin{equation*}
\underset{F,M }{\operatorname{Maximize}}   \operatorname{Tr}\left(M\right)\,.
\end{equation*}
Using the Schur complement property, the constraint~\eqref{dp1-1} in Problem~\ref{p3} can be rewritten  in the following  linear matrix inequality  (LMI)  form:
\begin{equation}\label{p4-2}
\left[\begin{array}{cc}
F_{11}-M & F_{12} \\
F_{12}^{\top} & F_{22}
\end{array}\right] \succeq 0\,.
\end{equation}	
Moreover,  the constraint~\eqref{dp1-3} in Problem~\ref{p3} is implicitly contained in~\eqref{p4-2}. 
Introducing the auxiliary matrix $\bar{I}\triangleq[\begin{array}{ll}
I_{n+m}& I_{n+m}
\end{array}]$ and  $\overline{D}\triangleq  [\begin{array}{cc}
A & B
\end{array}]\oplus [\begin{array}{ll}
A_{1} & B_{1}
\end{array}]$, and then using  the technique    of the matrix direct sum  and  the Schur complement property once more, the constraint~\eqref{dp1-2} in Problem~\ref{p3} can also be rewritten  in the  LMI  form  described by
\begin{equation*}
\left[\begin{array}{cc}
\alpha\bar{I} \bar{D}^{\top}\overline{F_{11}^{2}(\sigma)}
	\overline{D}\bar{I}^{\top}-F+W 
& \alpha^{\frac{1}{2}}\bar{I}\bar{D}^{\top} \overline{F_{12}^{2}(\sigma)} \\
\alpha^{\frac{1}{2}}  \overline{F_{12}^{2}(\sigma)}^{\top}\bar{D}\bar{I}^{\top} & \overline{F_{22}}^{2} \end{array}\right] \succeq 0,
\end{equation*}
where   $\overline{F_{ij}^{2}(\sigma)}\triangleq F_{ij}\oplus \sigma F_{ij}\,,i,j=1,2$, and $\overline{F_{22}}^{2}\triangleq F_{22}\oplus  F_{22}$.
Consequently, Problem~\ref{p3} is equivalently described by a standard SDP problem in Problem~\ref{p4}.
\begin{Problem}\rm \label{p4} SDP with optimization variables $ F \in \mathcal{S}^{n+m}$ and $M \in \mathcal{S}^{n}$.
	\begin{equation*}
	\begin{aligned}
	&	\underset{F,M }{\operatorname{Maximize}}  \ \operatorname{Tr}\left(M\right)\,, \\
	&	\text { subject to } \ 	 \left[\begin{array}{cc}
	F_{11}-M & F_{12} \\
	F_{12}^{\top} & F_{22}
	\end{array}\right] \succeq 0\,,	\\
	&\left[\begin{array}{cc}
	{\alpha\bar{I} \overline{D}^{\top}\overline{F_{11}^{2}(\sigma)}  
		\overline{D}\bar{I}^{\top}-F+W} 
	& \alpha^{\frac{1}{2}}\bar{I}\overline{D} ^{\top}  \overline{F_{12}^{2}(\sigma)} \\
	\alpha^{\frac{1}{2}}  \overline{F_{12}^{2}(\sigma)}^{\top}\overline{D}\bar{I}^{\top} & \overline{F_{22}}^{2} 	
	\end{array}\right] \succeq 0\,.
	\end{aligned}
	\end{equation*}
\end{Problem} 
 \begin{Remark}\rm	
	 Unlike the deterministic LQ problem studied in \cite{farjadnasab2022model}, this paper investigates the stochastic LQR problem for stochastic systems subject to multiplicative and additive noises. Consequently, the third constraint~(\ref{dp1-2}) in the dual problem differs from its counterpart in \cite{farjadnasab2022model} not only by adding a discount factor coefficient $\alpha$ but also by including a term related to multiplicative noise: $\alpha \sigma [\begin{array}{ll}
	 A_{1} & B_{1}
	 \end{array}]^{\top}\mathscr{P}(F)[\begin{array}{ll}
	 A_{1} & B_{1}
	 \end{array}] $. The presence of this term makes it impossible to rewrite~(\ref{dp1-2}) into LMI form solely by exploiting the Schur complement property. To address this, auxiliary matrices $\bar{I}$ and  $\overline{D}$ are introduced, and the technique of direct sum  is employed to achieve the desired transformation. This technique is also utilized in the subsequent reformulation of Problem~5. 
\end{Remark}

With the above preparation, we
can now present the  model-based SDP algorithm.
\begin{algorithm}[htb]
	\caption{Model-Based SDP Algorithm}
	\label{alg:0}
	\begin{algorithmic}[1]
	\State   Solve  SDP in Problem~\ref{p4}  for $\left( F^{*}, M^{*}\right)$.	
	\State  Obtain   the solution $P^{*} $ of the DGARE~\eqref{are1} as  $P^{*}=M^{*}$, and the optimal state feedback gain $ L^{*}=-\left( F^{*}_{22}\right) ^{-1} \left( F^{*}_{12}\right) ^{\top}$.
	\end{algorithmic}
\end{algorithm}	

\section{Model-free  implementation of model-based SDP algorithm } 
In this section, based on Monte-Carlo method,
a model-free implementation
of the SDP controller design is given,   and a completely data-driven optimal stochastic LQR control is presented by using the technique of the matrix direct sum.

Assume  that the   triplets $\left(	x_{k}^{(i)},u_{k}^{(i)},x_{k+1}^{(i)} \right)\,,i =  1,\cdots,N$, $k = 0,\cdots,K-1$ are available, where $N$ denotes the number of  sampling experiments implemented and  $K$ denotes the length of the sampling time  horizon in the data collection phase. 
Let $Z^{(1)}, Z^{(2)},\cdots,Z^{(N)}$ be a set of samples from $Z \triangleq [\begin{array}{llll}
z_{0} & z_{1} & \cdots & z_{K-1} \end{array}] $ and   $Y^{(1)}, Y^{(2)},\cdots,Y^{(N)}$ be a set of samples from $	Y \triangleq [\begin{array}{llll}
x_{1} & x_{2} & \cdots & x_{K} 
\end{array}]$.  According to \eqref{eq1}, we have  
\begin{equation}\label{y}
Y=	[\begin{array}{ll}
A & B
\end{array}] Z+[\begin{array}{ll}
A_{1} & B_{1}
\end{array}] Z\mathcal{V}+\mathcal{W},
\end{equation}  
where $\mathcal{V}\triangleq  v_{0}\oplus v_{1}\oplus \cdots\oplus v_{K-1}$, $\mathcal{W}\triangleq[\begin{array}{llll}
w_{0} & w_{1} & \cdots & w_{K-1}\end{array}]$. Noting \eqref{y} and using  Assumption~\ref{A1},  one has  
\begin{equation}\label{mk}
\begin{aligned}
&\mathbb{E}\left[ Y^{\top}\mathscr{P}(F) Y \right] \\
=&\mathbb{E}\left[Z^{\top}[\begin{array}{ll}
A & B
\end{array}]^{\top} \mathscr{P}(F)[\begin{array}{ll}
A & B
\end{array}]Z\right]\\
&+\sigma \mathbb{E}\left[Z^{\top}[\begin{array}{ll}
A_{1} & B_{1}
\end{array}]^{\top}\mathscr{P}(F)[\begin{array}{ll}
A_{1} & B_{1}
\end{array}] Z\right]\\
&+\mathbb{E}\left[\mathcal{W}^{\top}\mathscr{P}(F)\mathcal{W}\right]. 
\end{aligned}
\end{equation} 	
We now employ \eqref{mk} to provide the  model-free implementation of Algorithm~1. Assuming that $Z^{(i)}$ has full row  rank  for each $i =  1,\cdots,N$, then by the properties of matrix congruence,  if the constraint~\eqref{dp1-2} in Problem~\ref{p3} is left-multiplied by $ Z^{(i)\top}$ and right-multiplied by $Z^{(i)}$, we  obtain the equivalent constraint of \eqref{dp1-2} as follows: 
\begin{equation*}
\begin{aligned}
&\alpha Z^{(i)\top} [\begin{array}{ll}
A & B
\end{array}]^{\top}\mathscr{P}(F)[\begin{array}{ll}
A & B
\end{array}]Z^{(i)}\\  
+&\alpha\sigma  Z^{(i)\top}[\begin{array}{ll}
A_{1} & B_{1}
\end{array}]^{\top}\mathscr{P}(F)[\begin{array}{ll}
A_{1} & B_{1}
\end{array}] Z^{(i)}\\
-& Z^{(i)\top}\left( F-W\right) Z^{(i)}\succeq 0\,.
\end{aligned}
\end{equation*}	
Using the properties of positive semidefinite  matrices and $N>0$, one obtains 
\begin{equation}\label{qiwang}
\begin{aligned}
&\frac{1}{N}\sum_{i=1}^{N}\left[\alpha Z^{(i)\top}[\begin{array}{ll}
A & B
\end{array}]^{\top}\mathscr{P}(F)[\begin{array}{ll}
A & B
\end{array}] Z^{(i)}\right] \\
+&	\frac{1}{N}\sum_{i=1}^{N}\left[\alpha\sigma  Z^{(i)\top}[\begin{array}{ll}
A_{1} & B_{1}
\end{array}]^{\top}\mathscr{P}(F)[\begin{array}{ll}
A_{1} & B_{1}
\end{array}] Z^{(i)}  \right]\\ 
-&\frac{1}{N}\sum_{i=1}^{N}\left[ Z^{(i)\top}\left( F-W\right) Z^{(i)}\right] \succeq 0\,.
\end{aligned}
\end{equation}
Based on the Monte-Carlo method, we can approximate the mathematical expectation using the numerical average of $N$ sample paths. 
According to \eqref{mk},  one has 
\begin{equation}\label{p5-33}
\begin{aligned}
&\frac{1}{N}\sum_{i=1}^{N}\left[\alpha Z^{(i)\top}[\begin{array}{ll}
A & B
\end{array}]^{\top}\mathscr{P}(F)[\begin{array}{ll}
A & B
\end{array}] Z^{(i)}\right]\\
+&\frac{1}{N}\sum_{i=1}^{N}\left[ \alpha\sigma  Z^{(i)\top}[\begin{array}{ll}
A_{1} & B_{1}
\end{array}]^{\top}\mathscr{P}(F)[\begin{array}{ll}
A_{1} & B_{1}
\end{array}] Z^{(i)}\right]\\
\approx&\frac{1}{N}\sum_{i=1}^{N}\left[\alpha Y^{(i)\top}\mathscr{P}(F)Y^{(i)}\right]-\alpha \mathbb{E}\left[\mathcal{W}^{\top}\mathscr{P}(F)\mathcal{W}\right]\,.
\end{aligned}
\end{equation}
By substituting \eqref{p5-33}  into  \eqref{qiwang} and noting that 
\begin{equation*}
\mathbb{E}\left[\mathcal{W}^{\top}\mathscr{P}(F)\mathcal{W}\right]\succeq 0\,,
\end{equation*}
  we have
\begin{equation}\label{p5-3}
\frac{1}{N}\sum_{i=1}^{N}\left[\alpha Y^{(i)\top}\mathscr{P}(F)Y^{(i)}- Z^{(i)\top}\left( F-W\right) Z^{(i)}\right]
\succeq 0\,.
\end{equation}
Since $F_{22}\succ 0$, using the Schur complement and the technique of  matrix direct sum, the constraint~\eqref{p5-3} can be equivalently expressed as 
\begin{equation}\label{p5-66}
\begin{aligned}
&
\left[ \begin{array}{cc} \alpha\overline{I_{K}}^{N} \overline{Y}^{N\top}\overline{F_{11}}^{N}  
	\overline{Y}^{N}\overline{I_{K}}^{N\top}-\overline{Z}^{N}(F)
&\alpha^{\frac{1}{2}}\overline{I_{K}}^{N}\overline{Y}^{N\top} \overline{F_{12}}^{N} \\
\alpha^{\frac{1}{2}} \overline{F_{12}}^{N\top}\overline{Y}^{N}\overline{I_{K}}^{N\top} & \overline{F_{22}}^{N} 	
\end{array}\right] \\
&\succeq 0
\,,	
\end{aligned}
\end{equation}
where 
\begin{equation*}
\begin{aligned}
&\overline{F_{ij}}^{N}\triangleq \underbrace{F_{ij}\oplus F_{ij}\oplus\cdots \oplus F_{ij}}_{N}\,,i,j=1,2,\\
&\overline{I_{K}}^{N}\triangleq\left[\underbrace{\begin{array}{llll}
	I_{K}& I_{K} & \cdots & I_{K}
	\end{array}}_{N}\right],\\
&\overline{Y}^{N}\triangleq  Y^{(1)}\oplus Y^{(2)}\oplus\cdots \oplus Y^{(N)},\\
&\overline{Z(F)}^{N}\triangleq \sum_{i=1}^{N}\left[  Z^{(i)\top}\left( F-W\right) Z^{(i)}\right].
\end{aligned}
\end{equation*}

Based on the above analysis, Theorem~\ref{th3} introduces an SDP in Problem~\ref{p5} for solving the stochastic LQR problem without any information about the system dynamics and without requiring  an initial stabilizing  control policy. 
	\begin{Problem}\rm	\label{p5}
	SDP with optimization variables  $F \in \mathcal{S}^{n+m}$  and  $M \in \mathcal{S}^{n}$.
	\begin{align*}
	&	\underset{F,M }{\operatorname{Maximize}}  \ \operatorname{Tr}\left(M\right)\,, \notag\\
	&	\text { subject to } \	 \left[\begin{array}{cc}
	F_{11}-M & F_{12} \\
	F_{12}^{\top} & F_{22}
	\end{array}\right] \succeq 0\,,\notag\\		
	&
	\left[ \begin{array}{cc} \alpha\overline{I_{K}}^{N} \overline{Y}^{N\top}\overline{F_{11}}^{N}  
		\overline{Y}^{N}\overline{I_{K}}^{N\top}-\overline{Z}^{N}(F)
	&\alpha^{\frac{1}{2}}\overline{I_{K}}^{N}\overline{Y}^{N\top} \overline{F_{12}}^{N} \\
	\alpha^{\frac{1}{2}} \overline{F_{12}}^{N\top}\overline{Y}^{N}\overline{I_{K}}^{N\top} & \overline{F_{22}}^{N} 	
	\end{array}\right] \\
	&\succeq 0\,.\notag
	\end{align*} 
\end{Problem}
\begin{theorem}\rm \label{th3}
	Suppose that $Z^{(i)}$ has full row  rank  for each $i =  1,\cdots,N$.   Then the optimal state feedback gain associated with the stochastic LQR problem is approximated by
	\begin{equation}\label{lo}
	\widehat{L^{*}}=-\left(\hat F_{22}\right)^{-1}\left(\hat F_{12}\right)^{\top} \,,
	\end{equation}
	where  $\left( \hat F,\hat M\right)$  is the optimal point of  Problem~\ref{p5} with $\hat{F} = \begin{bmatrix} \hat{F}_{11} & \hat{F}_{12} \\ (\hat{F}_{12})^\T & \hat{F}_{22} \end{bmatrix}$.
\end{theorem} 
\begin{proof}
	It is already shown that according to the  Monte-Carlo method and  the  properties
	of matrix congruence and the  positive semidefinite  matrices,   the
		constraint~\eqref{dp1-2} in Problem~\ref{p3} can be approximated by the constraint~\eqref{p5-3}, and the constraint~\eqref{p5-3} is equivalent to 
		the constraint~\eqref{p5-66}. Based on Theorem~\ref{th2}, the  exact parameter $H^*$ of the optimal Q-function  is obtained by solving Problem~\ref{p3}, which is approximately equivalent to Problem~\ref{p5}, and then Q-learning offers  an estimate of the optimal state feedback gain,  as shown in~\eqref{lo}.\qed
\end{proof}
It should be noted that, to ensure the rank condition in Theorem~\ref{th3}, the length of the sampling time  horizon  $K$ must be more than or equal to the sum of the dimensions of the input and the state. The rank condition in Theorem~\ref{th3} is aimed at allowing  the system to oscillate sufficiently during the learning process to produce enough data to  fully learn the system information. To this end, in practice, a small exploration noise consisting of a Gaussian white noise or a sum of sinusoids must be added to the control signal.
\begin{Remark}\rm	
	The model-free implementation process in this paper differs from that in \cite{farjadnasab2022model}. Due to the presence of multiplicative and additive noise, the next-time-step state $x_{k+1}$ in this paper becomes a random variable, which is fundamentally different from the state-determined representation in \cite{farjadnasab2022model}. Therefore, the model-free implementation method from \cite{farjadnasab2022model} is no longer applicable to this paper. To address this challenge, we introduces a set of $N$ samples for the next-time-step state vector $	Y $ and the augmented system state vector $Z $ for  $k = 0,\cdots,K-1$. Based on the Monte-Carlo method, we approximate the mathematical expectation using the sample mean to derive an approximate equivalent description of the constraints. This enables the estimation of the optimal feedback gain.
\end{Remark}
In view of the results given above, Algorithm~\ref{alg:2} allows one to solve the stochastic LQR problem on
the basis of the available data.
\begin{algorithm}[htb]
	\caption{Model-Free SDP Algorithm}
	\label{alg:2}
	\begin{algorithmic}[1]
	\State  Choose the initial state $x_{0}$ from a  distribution with  the constant mean vector   and positive definite covariance matrix.
	\State   Apply  $u_k = u_{k}^{(i)} +d_k$  to system~\eqref{eq1} for $K$ time steps, where $d_k$ is the exploration noise with zero mean     and  covariance matrix $\Sigma_{d}\succ 0$, and collect the input and state data to construct the data matrices $Z^{(i)}$ and $Y^{(i)}$  for each $i =  1,\cdots,N$ until the rank condition in Theorem~\ref{th3} is satisfied.		
	\State   Solve SDP in Problem~\ref{p5}   for $\left( \hat F,\hat M\right)$.	
	\State  Obtain the  estimate of  $P^{*}$ given by $\widehat{P^{*}}=\hat M$, and the estimate of the optimal state feedback gain $ \widehat{L^{*}} $ associated with stochastic LQR problem  described by~\eqref{lo}.
	\end{algorithmic}
\end{algorithm}

It is worth noting that model-free iterative algorithms, such as policy iteration, value iteration, and Q-learning,  can also solve the stochastic LQR problem and   guarantee the  convergence to the optimal policy as the number of iterations goes to infinity.    
However, the proposed off-policy algorithm is non-iterative and computes the estimate of the optimal policy in a single step by utilizing only the sampled data with time-domain length $K$ ($K\ge n+m$). It neither requires an initial stabilizing controller to be provided, nor does it require hyper-parameter tuning. It is characterized by high sampling efficiency and, as a known feature of LMIs, it is inherently robust to model uncertainty.  Furthermore,  since the SDP in Problem~\ref{p5} does not contain terms associated with multiplicative and additive noises,  the implementation of the proposed model-free  SDP algorithm   does not require the multiplicative and additive noises to be measurable.

	\begin{Remark}\rm	
		The multiplicative noise is set to be a scalar $v_k$ just for simplicity of presentation, and the results of  this paper can be generalized to the following system:
		\begin{equation*}
		x_{k+1}=A x_{k}+B u_{k}+\sum_{l=1}^{M} \left( A_{l} x_{k}+B_{l} u_{k}\right) v^{l}_{k}+w_{k},
		\end{equation*}
		where  $\vec{v }_{k}=\left(v^{l}_{k},v^{2}_{k},\cdots, v^{M}_{k}\right) \in \mathcal{R}^{M}$ is  system multi-dimensional multiplicative noise. 	
		The following assumptions are made regarding the aforementioned system:		
		The noises $v^{l}_{k},v^{2}_{k},\cdots, v^{M}_{k}$  are  i.i.d. (not necessarily Gaussian distributed) satisfying  $\mathbb{E}(v^{l}_{k}) =0$ and $\mathbb{E}\left[v^{l}_{k}v^{l\top}_{k}\right]  =\delta_{ij }\sigma$, where $\sigma > 0$ and $\delta_{ij}$ is a Kronecker function. Moreover, $\left\{w_{k}, k \in \mathcal{Z}_{+} \right\}$ and $\left\{\vec{v}_{k}, k \in \mathcal{Z}_{+} \right\}$ are mutually independent. 		
		If multiplicative noise is multidimensional, e.g., $\vec{v }_{k}=\left(v^{l}_{k},v^{2}_{k},\cdots, v^{M}_{k}\right)$,  the augmented system
		becomes
		\begin{equation*}
		z_{k+1}=\overline{A_{L}} z_{k}+\sum_{l=1}^{M}\overline{A^l_{L}} z_{k}v^l_{k}+\overline{L} w_{k}\,,
		\end{equation*}
		where  $\overline{A^l_{L}} \triangleq \left[\begin{array}{cc}A_{l} & B_{l} \\ L A_{l} & L B_{l}\end{array}\right] \in \mathcal{R}^{(n+m) \times(n+m)} $, the remaining symbols are the same as when  $v_k$ a scalar. Accordingly, Equation~(8) satisfied by $S$ in this paper becomes
		\begin{equation*}
		\begin{aligned}
		&\alpha  \left( \overline{A_{L}} S \overline{A_{L}}^{\top}+\sigma\sum_{l=1}^{M}\overline{A^l_{L}} S \overline{A^l_{L}} ^{\top}\right) +\overline{L} X_{0} \overline{L}^{\top}\\
		+&\frac{\alpha}{1-\alpha} \overline{L}\Sigma \overline{L}^{\top}=S.
		\end{aligned}		
		\end{equation*}
		The corresponding results in  this paper can be proved in the same way
		for the multidimensional case. For example: Problem 3 needs only one modification,
		i.e., Equation~(12) is changed to be  
		\begin{equation*}
		\begin{aligned}
		&\alpha  [\begin{array}{ll}
		A & B
		\end{array}]^{\top}\mathscr{P}(F)[\begin{array}{ll}
		A & B
		\end{array}] +\alpha \sigma\sum_{l=1}^{M} [\begin{array}{ll}
		A_{l} & B_{l}
		\end{array}]^{\top}\mathscr{P}(F)[\begin{array}{ll}
		A_{l} & B_{l}
		\end{array}]\\
		-&F+W  \succeq 0,
		\end{aligned}
		\end{equation*}
		and the second constraint in Problem 4 is modified to
		\begin{equation*}
		\mathscr{T}(F)\triangleq\left[\begin{array}{cc}
		\mathscr{T}(F)_{11} 
		& \mathscr{T}(F)_{12}  \\
		\mathscr{T}(F)_{12} ^{\top} & \mathscr{T}(F)_{22}  	
		\end{array}\right] \succeq 0,
		\end{equation*}
		where 
		\begin{equation*}
		\begin{aligned}
		&\mathscr{T}(F)_{11}\triangleq\alpha\overline{I_{n+m}^{M+1}} \overline{D^{M+1}}^{\top}\overline{F_{11}^{M+1}(\sigma)} 
		\overline{D^{M+1}}\overline{I_{n+m}^{M+1}}^{\top}-F+W,\\
		&\mathscr{T}(F)_{12}\triangleq \alpha^{\frac{1}{2}}\overline{I_{n+m}^{M+1}}\overline{D^{M+1}} ^{\top}  \overline{F_{12}^{M+1}(\sigma)}, \\
		&\mathscr{T}(F)_{22}\triangleq \overline{F_{22}^{M+1}}
		\end{aligned}
		\end{equation*}
		with 
		\begin{equation*}
		\begin{aligned}
		&\overline{I_{n+m}^{M+1}}\triangleq\left[\underbrace{\begin{array}{llll}
			I_{n+m}& I_{n+m} & \cdots & I_{n+m}
			\end{array}}_{M+1}\right],\\
		&\overline{D^{M+1}}\triangleq  [\begin{array}{cc}
		A & B
		\end{array}]\oplus_{l=1}^{M} [\begin{array}{ll}
		A_{l} & B_{l}
		\end{array}],\\
		&\overline{F_{ij}^{M+1}(\sigma)}\triangleq F_{ij}\oplus_{l=1}^{M} \left( \sigma F_{ij}\right)\,,i,j=1,2,\\
		&\overline{F_{22}^{M+1}}\triangleq \oplus_{l=1}^{M+1}F_{22}.
		\end{aligned}
		\end{equation*}
	\end{Remark}

 We introduce five key optimization problems (Problems 1 to 5) in this paper. To improve readability, we visually illustrate their relationships in Fig.~1.
 \begin{figure}[htb]
	\begin{center}
		\includegraphics[width=\columnwidth]{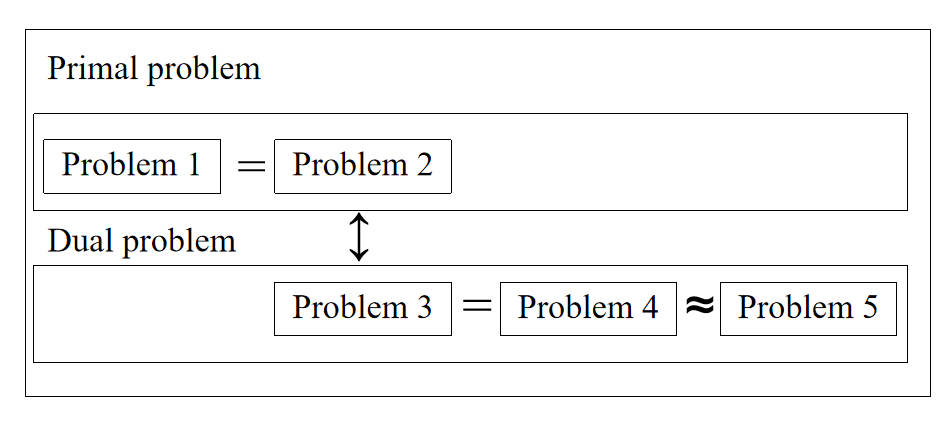}
		\caption{Diagram for relations among the results.}  
		\label{fig0}                                 
	\end{center}                                 
\end{figure}

In the aforementioned model-free implementation, replacing the mathematical expectation with sample averaging introduces error. In fact, due to the presence of stochastic noise, error is unavoidable. Beyond Monte Carlo approximation error, sources of error include but are not limited to
the impact of probing noise in data collection, the sample estimation error and numerical errors in SDP solution. Taking into account the presence of errors, it is
necessary to study the robustness of the model-free  algorithm to errors  in the next section.
\section{Robustness analysis}
In this section, the robustness of the model-free SDP algorithm to errors in the learning
process is studied. The estimation error  is
shown to be bounded and the estimation of the optimal control policy is proven to
be admissible under mild conditions.	
	\begin{theorem}\rm
	Assume  that the data matrices \( Z^{(i)} \in \mathbb{R}^{(n+m) \times K} \) (for \( i=1,2,\dots,N \)) constructed from system interaction data are full-row rank, with \( \sigma_{\text{min}}(Z^{(i)}) \geq \gamma > 0 \) and \( K \geq n+m \), and that there exists a constant $\underline{\sigma}$ such that \( \sigma_{\text{min}}(I - \alpha C_{\widehat{L^{*}}}) \geq \underline{\sigma} > 0 \),  where
	\begin{equation*}
	C_{\widehat{L^{*}}} = (A + B\widehat{L^{*}}) \otimes (A + B\widehat{L^{*}}) + \sigma (A_1 + B_1\widehat{L^{*}}) \otimes (A_1 + B_1\widehat{L^{*}}).
	\end{equation*}		
	Then, for any confidence level $1-\beta$ (\( \beta \in (0,0.1) \)), the estimation error satisfies
	\begin{equation*}
	\mathbb{P}\left( \|\widehat{L}^* - L^*\|_F \leq \widetilde{O}\left( \frac{(n+m)(1 + \|F^*\|_F)}{\sqrt{NK} \cdot \sigma_d \cdot \underline{\sigma}} \right) \right) \geq 1 - \delta,
	\end{equation*}
	where  \( \sigma_d > 0 \) is the minimum intensity of the exploration noise.
\end{theorem}

\begin{proof}
	The proof proceeds in four key steps, leveraging convex optimization duality, matrix perturbation theory, and probability tail bounds.
	
	Step 1: Probability Bound of Sample Average Deviation 
	
	Define the sample average deviation as
	\begin{equation*}
	\begin{aligned}
	\Delta_N = &\frac{1}{N} \sum_{i=1}^N \left[\alpha Y^{(i)\T} \mathscr{P}(\hat{F}) Y^{(i)} - Z^{(i)\T}(\hat{F} - W) Z^{(i)}\right] \\
	&- \mathbb{E}\left[ \alpha (Y^{(i)})^\top \mathscr{P}(\hat{F}) Y^{(i)} - (Z^{(i)})^\top(\hat{F}-W)Z^{(i)} \right].
	\end{aligned}
	\end{equation*}
	By Assumption 1, \( Y^{(i)} \) and \( Z^{(i)} \) are composed of independent and identically distributed random variables, so the elements of \( \Delta_N \) are sub-exponential random variables.	
	By the matrix-valued sub-exponential tail bound \cite{vershynin2018high} and Union Bound, we have
	\begin{equation*}
	\mathbb{P}\left( \|\Delta_N\|_F \leq \epsilon_1 \right) \geq 1 - \frac{\beta}{2},
	\end{equation*}
	where \( \epsilon_1 = c_1 \cdot \frac{(n+m)\ln(2/\beta)}{\sqrt{NK} \cdot \sigma_d \cdot \gamma} \) for some constant \( c_1 > 0 \) dependent on the noise covariance matrices \( \Sigma \) and \( \sigma \). The term \( 1/\sqrt{NK} \) arises from the law of large numbers, \( (n+m) \) corrects for the dimension of \( Z^{(i)} \), and \( \sigma_d \cdot \gamma \) reflects the persistent excitation of the data.
	
	By the expectation decomposition \eqref{mk}, ignoring the non-negative additive noise expectation term \( \mathbb{E}[\mathcal{W}^\top \mathscr{P}(F) \mathcal{W}] \), the sample average approximates the expectation, ensuring the covariance of \( \Delta_N \) is bounded by \( \sigma_d^2 K \), i.e., $\text{Tr}\left( \text{Cov}(\Delta_N) \right) \leq \sigma_d^2 K$. (For the definition of the covariance matrix of a random variable with matrix values, please refer to \cite{vershynin2018high}.)
	
	Step 2: Linear Relationship Between \( \Delta F \) and \( \Delta_N \)
	
	Let \( \Delta F = \hat{F} - F^* \) denote the deviation between the model-free optimal dual variable \( \hat{F} \) and the true optimal dual variable \( F^* = H^* \). Define the expectation term as 
	\begin{equation*} 
	G(F) = \mathbb{E}\left[ \alpha (Y^{(i)})^\top \mathscr{P}(F) Y^{(i)} - (Z^{(i)})^\top(F-W)Z^{(i)} \right].
	\end{equation*} 
	Since \( F^* \) is the optimal solution of Problem 4, it satisfies the stationary condition $G(F^*) = 0	$.		
	For small \( \Delta F \) (valid when \( N,K \) are sufficiently large), we perform a first-order Taylor expansion of the expectation term around \( F^* \):
	\begin{equation*}
	G(F)_{F=\hat{F}} \approx \nabla_F G(F)_{F=F^*} \cdot \Delta F.
	\end{equation*}		
	By the definition of \( \Delta_N \) and the law of large numbers, substituting the Taylor expansion gives
	\begin{equation*}
	\Delta_N \approx \nabla_F G(F)_{F=F^*} \cdot \Delta F.
	\end{equation*}
	Split $G(F)$ into quadratic and linear terms. Taking  the partial derivatives of the quadratic terms in blocks and combining them with the derivatives of the linear terms, the gradient \( \nabla_F G(F) \) is exactly the linear combination of the covariance evolution terms, leading to:
	\begin{equation*}
	\nabla_F G(F)_{F=F^*} = I - \alpha C_{\widehat{L^{*}}}.
	\end{equation*}		
	Thus, the linear equation holds:
	\begin{equation*}
	(I - \alpha C_{\widehat{L^{*}}}) \cdot \Delta F = \Delta_N.
	\end{equation*}		
	Step 3: Bound of Dual Variable Deviation \( \|\Delta F\|_F \)
	
	Taking the Frobenius norm of both sides of the linear equation and using the matrix norm inequality \( \|AX\|_F \geq \sigma_{\text{min}}(A) \cdot \|X\|_F \), we have:
	\begin{equation*}
	\sigma_{\text{min}}(I - \alpha C_{\widehat{L^{*}}}) \cdot \|\Delta F\|_F \leq \|\Delta_N\|_F.
	\end{equation*}		
	By the stability condition \( \sigma_{\text{min}}(I - \alpha C_{\widehat{L^{*}}}) \geq \underline{\sigma} > 0 \), rearranging gives:
	\begin{equation*}
	\|\Delta F\|_F \leq \frac{\|\Delta_N\|_F}{\underline{\sigma}}.
	\end{equation*}		
	Combining with the probability bound of \( \Delta_N \) from Step~1, the Union Bound yields
	\begin{equation*}
	\mathbb{P}\left( \|\Delta F\|_F \leq \frac{\epsilon_1}{\underline{\sigma}} \right) \geq 1 - \beta.
	\end{equation*}		
	Step 4: Bound of Gain Estimation Error \( \|\Delta L\|_F \)
	
	Let \( \Delta L = \widehat{L}^* - L^* \), with \( L^* = -(F_{22}^*)^{-1}(F_{12}^*)^\top \) and \( \widehat{L}^* = -(\hat{F}_{22})^{-1}(\hat{F}_{12})^\top \). Let \( \Delta F_{12} = \hat{F}_{12} - F_{12}^* \) and \( \Delta F_{22} = \hat{F}_{22} - F_{22}^* \). Expanding \( \Delta L \) using the matrix inverse perturbation formula gives
	\begin{equation*}
	\Delta L = -(\hat{F}_{22})^{-1}\Delta F_{12}^\top + (\hat{F}_{22})^{-1}\Delta F_{22} \cdot (F_{22}^*)^{-1}(F_{12}^*)^\top.
	\end{equation*}		
	Since \( F_{22}^* \succ 0 \), for sufficiently small \( \|\Delta F_{22}\|_F \), \( \hat{F}_{22} \) is invertible with \( \|(\hat{F}_{22})^{-1}\|_2 \leq \frac{2}{\sigma_{\text{min}}(F_{22}^*)} \). Using the triangle inequality \( \|A+B\|_F \leq \|A\|_F + \|B\|_F \) and operator norm inequality \( \|AB\|_F \leq \|A\|_2 \|B\|_F \), we obtain
	\begin{equation}\label{eb}
	\|\Delta L\|_F \leq c \cdot \|\Delta F\|_F,
	\end{equation}
	where \( c = \frac{2(1 + 2\|L^*\|_2/\sigma_{\text{min}}(F_{22}^*))}{\sigma_{\text{min}}(F_{22}^*)} \) is a constant independent of \( N,K \).
	
	Substituting the bound of \( \|\Delta F\|_F \) from Step 3 and suppressing the logarithmic factor \( \ln(2/\beta) \) with \( \widetilde{O} \), we get
	\begin{equation*}
	\|\widehat{L}^* - L^*\|_F \leq \widetilde{O}\left( \frac{(n+m)(1 + \|F^*\|_F)}{\sqrt{NK} \cdot \sigma_d \cdot \underline{\sigma}} \right).
	\end{equation*}		
	This completes the proof.
\end{proof}

\begin{Remark} \rm
The error bound shows that the estimation error decays at a rate of $\mathcal{O}(1/\sqrt{NK})$, which is consistent with the statistical efficiency of data-driven control algorithms \cite{lai2023model}.  Moreover, based on the bounds for the estimation error given in Theorem~4, the sample complexity of Algorithm~2 can be analyzed, i.e., the total sample count $s = NK$ satisfies $s = \tilde{O}\left( \frac{(n+m)^2}{\varepsilon^2} \right)$. Unlike \cite{lai2023model}, the proof method in this paper transforms the error into a linear system transfer problem via dual variables. It unifies error boundedness and system stability analysis through the closed-loop matrix minimum singular value, explicitly quantifying the relationship between key factors (such as sample complexity and data sustained excitation) and the error. This approach eliminates reliance on Q-function modeling.
\end{Remark}
\begin{theorem}\rm	\label{th5}
	 Assume  that the   error between the estimated dual optimal solution \( \hat F \) and the true dual optimal solution \( F^* = H^* \)  satisfies $\normF{\hat{F} - F^*} \leq \varepsilon$, where \( \varepsilon> 0 \) is a sufficiently small constant depending on \( N \) and \( K \) (i.e., \( \varepsilon \to 0 \) as \( N, K \to \infty \)), then the control policy $ u_k = \widehat{L^{*}}x_k$ is admissible.
\end{theorem}

\begin{proof} 
For the true optimal gain $L^*$, since $L^{*} \in \mathcal{L}$, \cite[Lemma~1]{lai2023model} implies that
 $\rho(C_{L^*}) < 1$.

For the estimated gain $\widehat{L^{*}}$, according to~\eqref{eb}, the estimation error $\normF{\hat{F} - F^*} \leq \varepsilon$  implies
\begin{equation*}
\normF{\widehat{L^{*}} - L^*} \leq c\varepsilon .
\end{equation*} 
That is, there exists a $\delta = c\varepsilon$ such that $\widehat{L^{*}} \in \mathcal{B}_{\delta}\left(L^{*}\right)\triangleq\left\lbrace L \in \mathcal{R}^{m \times n}\Big| \|L-L^{*}\|_{F} \leqslant \delta\right\rbrace$, where $c$ is a positive constant given in Theorem~4. From the continuity of matrix operations and the continuity of the spectral radius in \cite{boyd2004convex}, we obtain $\rho(C_{\widehat{L^{*}}}) < 1$, for all $\widehat{L^{*}} \in \mathcal{B}_{\delta}\left(L^{*}\right)$.
 According to \cite[Lemma~1]{lai2023model}  and Definition~2, the control policy $u_k = \widehat{L^{*}}x_k$ is admissible. This completes the proof. \qed
\end{proof}
\begin{Remark}\rm
	Theorem 5 aligns with Theorem 3 in \cite{lai2023model}  in core logic: both rely on the continuity of stability with respect to control gain and the boundedness of estimation error. The key distinction lies in the gain estimation source—$\widehat{L^{*}}$ in this paper is derived from non-iterative SDP optimization rather than iterative RL, but the admissibility proof leverages the same equivalence between spectral radius condition and MSS, and the small estimation error assumption. This confirms that Algorithm 2  not only achieves high estimation accuracy but also preserves closed-loop stability, ensuring practical applicability.
\end{Remark}

\section{Simulation}
Consider the pulse-width modulated inverter system used in the  uninterruptible power supply studied in \cite{jung1996discrete},  and assume that
it is perturbed by both multiplicative and additive noises. Then the  discrete-time dynamical system is obtained via   Euler discretization with
step size $\delta t = 0.1s$,  which can be described by system~\eqref{eq1} with  
$A=\left[\begin{array}{cc}
0.6929 & 8.6545\\
-0.0241 & 0.8603	
\end{array}\right]$, $A_{1}=\left[\begin{array}{cc}
0.01 & 0.02 \\
-0.001 & 0.05 	
\end{array}\right]$, $B=\left[\begin{array}{cc}
0.1290  &	0.0267   	
\end{array}\right]^{\top}$, and $B_1=\left[\begin{array}{cc}
-0.02  &	0.005   	
\end{array}\right]^{\top}$.
Assume that $v_{k}$ and  $w_{k}$ follow the standard Gaussian distributions. 
Select $Q=I$, $R=0.00001$ in the cost   functional~\eqref{j1}. Choose $\alpha=0.5$. Then,   the optimal feedback control gain
$L^{*} $ can be computed by the model-based SDP Algorithm~1     as
$  L^{*}=\left[\begin{array}{ll}
-4.8599 & -64.0491
\end{array}\right]$,
and   the solution $P^{*} $ of DGARE~\eqref{are1}   can be obtained by $P^{*}=M^{*} $ in Theorem~2 as
$  P^{*}=\left[\begin{array}{ll}
1.0215   & 0.1206\\
0.1206  &  1.6917
\end{array}\right]$,  with the residual $\|e(P^{*}) \|_{F} =5.8025\times10^{-4}$, where $e(P) \triangleq P-\mathscr{R}(P) $.

The model-free Algorithm~\ref{alg:2} is employed to     estimate  the optimal   feedback gain
in the steps that follow. Choose the initial state $x_{0}$ from a Gaussian distribution with   mean vector  $\mu_{0}=[1, 2]^{\top}$ and  covariance matrix  $\Sigma_{0}= 5I_2 \succ 0 $. Let  $K  = 9$.  Select $5$ values of $N$, that is, 10, 20, 30,
40 and  80. For each value of
$N$, Algorithm~\ref{alg:2} is run for 10  times, in which $w_k$   and $v_k$
are generated randomly for each time.  Then,   CVX is used to solve Problem~\ref{p5} each time, and the estimates of $L^{*}$ and $P^{*}$ can be obtained, which are denoted as $\widehat{L^{*}}$ and $\widehat{P^{*}}$, respectively. Fig.~\ref{fig1} shows the residual $\|e(\widehat{P^{*}}) \|_{F}$
of Algorithm~\ref{alg:2} with different values
of $N$. One can observe   that the average  of  $\|e(\widehat{P^{*}}) \|_{F}$ decreases   as the value of $N$ increases. When $N=20$, the  average of $\|e(\widehat{P^{*}}) \|_{F}$  can achieve an order of  $10^{-4}$.  One can observe from the
simulation results that, although Algorithm~\ref{alg:2} does not know the information about the system dynamics, the estimate of the optimal control gain obtained is very close to the optimal control gain obtained from the model-based SDP Algorithm~1, and the model-free Algorithm~\ref{alg:2} performs very effectively.
\begin{figure}
	\begin{center}
		\includegraphics[height=2in,width=3.5in]{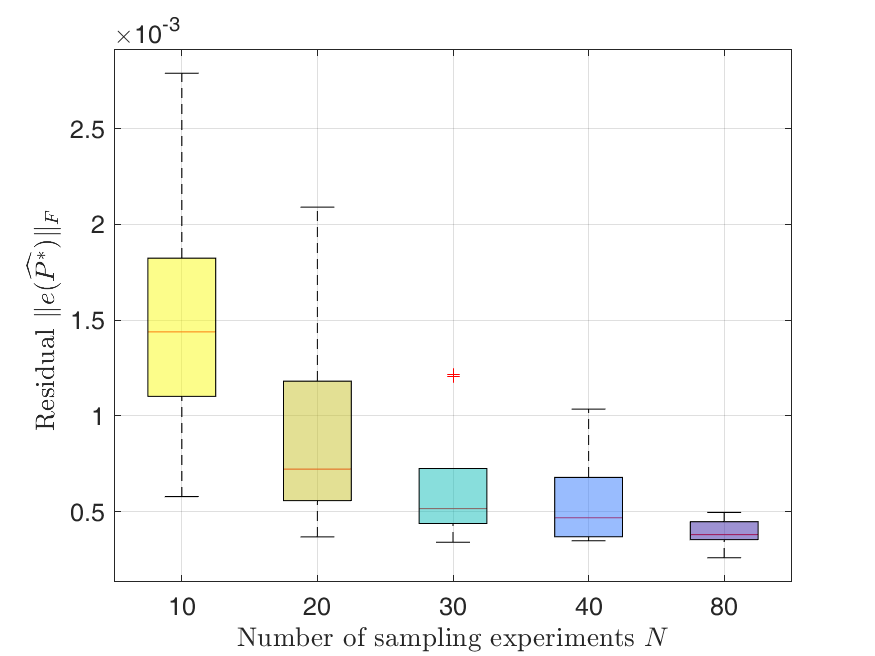}    
		\caption{Residuals $\|e(\widehat{P^{*}}) \|_{F}$ of Algorithm~2 with different values of $N$.}  
		\label{fig1}                                 
	\end{center}                                 
\end{figure}
\begin{figure}
	\begin{center}
		\includegraphics[height=2in,width=3.5in]{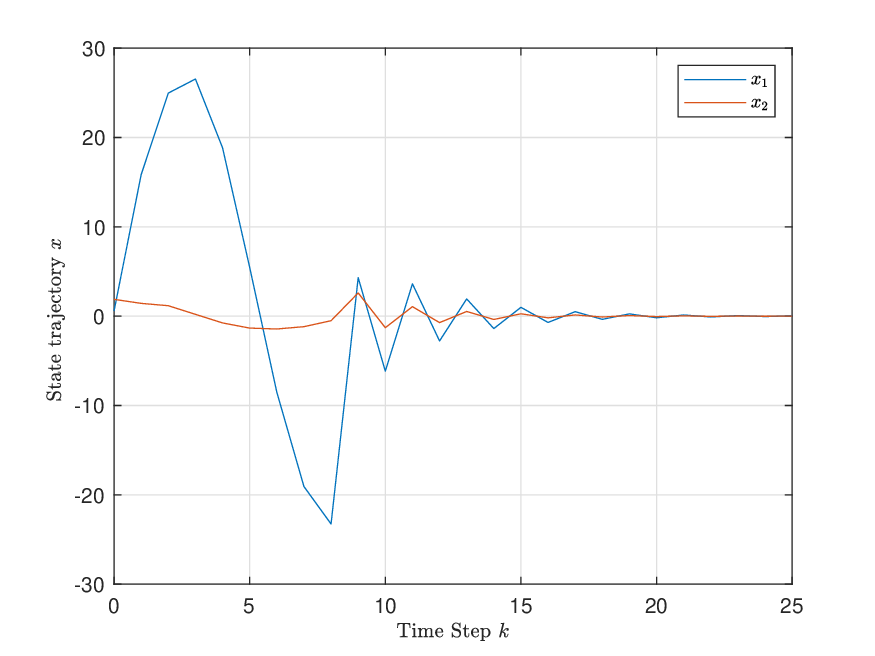}    
		\caption{A numerical average of state trajectories in data collection phase and after learning.}  
		\label{fig3}                                 
	\end{center}                                 
\end{figure}
Fig.~\ref{fig3} shows the state trajectories derived from the numerical average of the data collected from $N=80$ sampling experiments during the model-free design approach, where the trajectories in time domains $\left\lbrace 0,\cdots,8 \right\rbrace $ are the state trajectories gathered during the data collection phase of the reinforcement learning process, and starting at moment $9$, the system begins to apply the model-free optimal controller to generate optimal state trajectories. The results confirm that the proposed model-free algorithm is able to  efficiently
learn the optimal control gain and stabilize the system. 
\section{Conclusions}
In this paper, a novel model-free SDP method for solving
the stochastic LQR problem  has been proposed. The controller design scheme is based on  an SDP  with LMI
constraints,  which is sample-efficient and non-iterative. Moreover, it neither requires an initial stabilizing controller to be provided, nor does it require multiplicative and additive noises to be measurable. The robustness of the model-free SDP method to errors  has been investigated. In addition, the proposed algorithm  has been verified on a pulse-width modulated inverter system. The results confirm that the proposed algorithm is able to  efficiently
learn the optimal control gains and stabilize the systems.
\bibliographystyle{apacite} 
\bibliography{gjref.bib}           
\end{document}